\documentclass[11pt]{amsart}
\usepackage{amsmath,amsmath,latexsym,amssymb,amsfonts,amsbsy, amsthm}
\usepackage{mathrsfs}
\usepackage{fancyhdr}
\usepackage{latexsym}
\usepackage{multicol,graphics}
\usepackage{amssymb}
\usepackage{graphicx}
\usepackage{color}
\usepackage{bm}
\allowdisplaybreaks

\newcommand{\R}{{\mathbb R}}

\newtheorem{theorem}{Theorem}[section]
\newtheorem{lemma}[theorem]{Lemma}
\newtheorem{proposition}[theorem]{Proposition}
\newtheorem{remark}[theorem]{Remark}

\newtheorem{definition}{Definition}[section]

\numberwithin{equation}{section}

\begin{document}

\title[]{Low regularity solutions  of non-homogeneous boundary value problems of a higher order Boussinesq equation in a quarter plane}
\author[]{Shenghao Li}
\author[]{Min Chen}
\author[]{Bingyu Zhang}
\date{}
\begin{abstract}
In this article, we study an initial-boundary-value problem of the sixth order Boussinesq equation on a half line
with nonhomogeneous boundary conditions:
\begin{equation*}
\begin{cases}
u_{tt}-u_{xx}+\beta u_{xxxx}-u_{xxxxxx}+(u^2)_{xx}=0,\quad x>0\mbox{, }t>0,\\
u(x,0)=\varphi (x), u_t(x,0)=\psi ''(x), \\
u(0,t)=h_1(t), u_{xx}(0,t)=h_2(t), u_{xxxx}(0,t)=h_3(t),
\end{cases}
\end{equation*}
where $\beta=\pm1$. It is shown that the problem is locally well-posed  in $H^s(\R^+)$ for $-\frac12<s\leq 0$ with initial condition $(\varphi,\psi)\in H^s(\R^+)\times H^{s-1}(\R^+)$ and boundary condition
$(h_1,h_2,h_3) $ in the product space $H^{\frac{s+1}{3}}(\R^+)\times H^{\frac{s-1}{3}}(\R^+)\times H^{\frac{s-3}{3}}(\R^+)$.
\end{abstract}

\keywords{boundary value problem; Bourgain space; sixth order Boussiensq equation. }


\maketitle


\section{Introduction}
\setcounter{equation}{0}
\setcounter{theorem}{0}
Consideration is given to the sixth order Boussinesq equation,
\[u_{tt}-u_{xx}+\beta u_{xxxx}-u_{xxxxxx}+(u^2)_{xx}=0, \quad \beta=\pm1.\]
 It is derived by Christov, Maugin and Velarde \cite{28} as a model for  surface long water waves    with small amplitude in a channel with flat bottom and it is also proposed in modeling the nonlinear lattice dynamics in elastic crystals by Maugin \cite{66}.

 The sixth order Boussinesq equation has structural similarities related to both the linear ``good'' Boussinesq equation
 \begin{equation*}
u_{tt}-u_{xx} + u_{xxxx}=0,
\end{equation*}
 and the linear KdV-type equation
  \begin{equation*}
u_{t}\pm u_{xxx}=0,
\end{equation*}
where the relation to the KdV-type equation can be seen by considering its the linearized equation as
 \begin{equation*}
    u_{tt}-u_{xxxxxx}:=(\partial_t+\partial_{xxx})(\partial_t-\partial_{xxx})u=0.
\end{equation*}
 Both of the Boussinesq equation and the KdV equation have been well-studied during the past decades \cite{4,6,7,12,17,24,32,40,41,42,54,53,52,50,59,60,64,63,62,61,67,75,76,79,84,82,83,81,85} and these work has been contributed well to the research on the sixth order Boussiensq equation such as its initial-value problem (IVP) posed on $\omega$,
\begin{equation}\label{ivp}
    \begin{cases}
     u_{tt}-u_{xx}+(u^2)_{xx}+\beta u_{xxxx}-u_{xxxxxx}=0,\ \ \ x\in \omega,\\
     u(x,0)=\varphi (x), u_t(x,0)= \psi''(x).
    \end{cases}
\end{equation}
 For $\omega= \R$, Esfahani, Farah and Wang \cite{35} have shown the local well-posedness in $H^s(\R)$  for initial data  $(\varphi,\psi)\in  H^s(\mathbb{R})\times H^{s-1}(\mathbb{R})$ with $s=0,1$, by applying the Strichartz type smoothing inherited from Linares' work \cite{59} for the Boussinesq equation;  Esfahani and Farah \cite{34} have then improved the result to  $s>-\frac{1}{2}$ by using a related   Bourgain space  inherited from Kenig, Ponce and Vega's work \cite{50} on the KdV equation; later on, they have proven the conclusion for $s>-\frac{3}{4}$ (c.f. \cite{37}) by using the $[k;Z]$-multiplier norm method for the KdV equation. In addition, its well-posedness on a periodic domain, $\omega=\mathbb{T}$, is shown by Wang and Esfahani \cite{79} for $s>-\frac{1}{2}$.

 However, the theory for the initial-boundary-value problem (IBVP) of the sixth order Boussinesq equation has been less developed. 
  Li, Chen and Zhang have studied its well-posedness on a half plane \cite{96} and a finite domain \cite{98} for $s\geq 0$. In this article, we continue our study for the IBVP of the sixth order Boussinesq equation on a half plane,
\begin{equation}\label{0-1}
\begin{cases}
u_{tt}-u_{xx}+\beta u_{xxxx}-u_{xxxxxx}+(u^2)_{xx}=0,\hspace{0.3in}\mbox{for }x>0\mbox{, }t>0,\\
u(x,0)=\varphi (x), u_t(x,0)=\psi ''(x), \\
u(0,t)=h_1(t), u_{xx}(0,t)=h_2(t), u_{xxxx}(0,t)=h_3(t).\\
\end{cases}
\end{equation}

\begin{figure}[!h]
\centering
   \includegraphics[width=10cm]{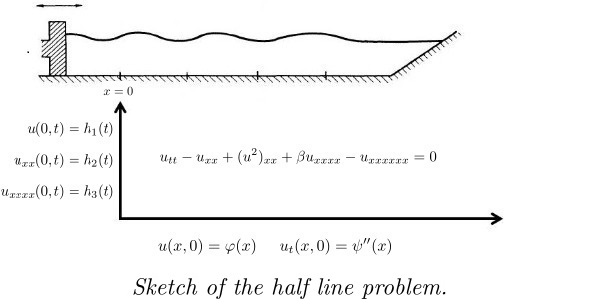}\\
\end{figure}

The figure  shows the model when the sixth order Boussinesq equation is applied to a physical problem. It is a model of unidirectional wave in an open uniform channel with a wave-maker at one end.  The $x$-axis denotes the distance from the wavemaker, the $t$-axis denotes the elapsed time and the dependent function $u(x,t)$ denotes the velocity of fluid along the $x$-direction at certain layer of the channel at position $x$ and on time $t$.

 We   aim  to establish the local well-posedness for the IBVP
 (\ref{0-1}) in the classic $L^2-$based Sobolev spaces, $H^s(\mathbb{R}^+)$, when the initial data  lies in $H^s(\mathbb{R}^+)\times H^{s-1}(\mathbb{R}^+)$ and the boundary   data  $(h_1,h_2,h_3)$ is drawn from the product space
 \begin{equation}\label{req}
 H^{s_1}(0,T)\times H^{s_2}(0,T)\times H^{s_3}(0,T)
\end{equation}
for some appropriate $s_1$, $s_2$ and $s_3$ depending on $s$. The following two questions  arise naturally.
\begin{itemize}
  \item  While the initial data  $\varphi $ and $\psi$ are required  to be in the spaces $H^s (\mathbb{R}^+)$ and $\ H^{s-1}(\mathbb{R}^+)$, respectively,  for the well-posedness of the IBVP  (\ref{0-1}) in the space $H^s (\mathbb{R}^+)$, what are the optimal  (minimum) regularity requirements  on the boundary data $(h_1,h_2,h_3)$ for the IBVP (\ref{0-1})?
  \item What is the smallest value  of $s$ such that the IBVP (\ref{0-1}) is  well-posed in the space $H^s (\mathbb{R}^+)$?
\end{itemize}
 We notice that the solution $u$ of the Cauchy problem for the linear sixth order Boussiensq equation,
 \begin{equation*}
    \begin{cases}
     u_{tt}-u_{xx}+\beta u_{xxxx}-u_{xxxxxx}=0,\ \ \ x\in \mathbb{R},\\
     u(x,0)=\varphi (x), u_t(x,0)= \psi''(x),
    \end{cases}
\end{equation*}
   possesses the sharp Kato smoothing property for any $s$:
 \[ \sup_{x\in \mathbb{R}} \| \partial ^j_x u(x, \cdot)\|_{H_t^{\frac{s-j+1}{3}} (\R^+) } \lesssim  \|\varphi\|_{H^s (\R)} + \|\psi\|_{H^{s-1} (\R)}, \quad  j=0, 1,...,5,\]
it is  therefore necessary  to have:
 \begin{equation} \label{x-1}
 h_1\in H^{\frac{s+1}{3}}_{loc} (\R^+), \quad h_2 \in H^{\frac{s-1}{3}}_{loc} (\R^+), \quad h_3 \in H^{\frac{s-3}{3}}_{loc} (\R^+)
  \end{equation}
  for the IBVP (\ref{0-1})  being well-posed in the space $H^s (\R^+)$.   It  has already been shown that (\ref{x-1}) is  also sufficient and the related IBVP  is indeed well-posed for $s\geq 0$ (c.f. \cite{97}). Thus, it is  natural to  expect similar conditions as in \eqref{x-1} still held for the IBVP (\ref{0-1})
    for $s\leq 0$. In this article, we   use the Bourgain-type space (c.f \cite{23,24,54,53,52,50,51}) to establish the local well-posedness for $-\frac12<s\leq0$.

 Before stating the main theorem, let us denote $\langle x\rangle=\sqrt{1+x^2}$ and introduce the related Bougain-type space, $X^{s,b}$, for the sixth order Boussinesq equation.
\begin{definition}
For $s,b\in\R$, $X^{s,b}$ denotes the completion of the Schwartz class ${\mathcal S}(\R^2)$ with
\[\|u\|_{X^{s,b}(\R^2)}=\left\|\langle\xi\rangle^s\langle|\tau|-\phi(\xi)\rangle^b\widehat{u}(\xi,\tau)\right\|_{L^2_{\xi,\tau}(\R^2)},\]
where $\phi(\xi)=\sqrt{\xi^6+\beta \xi^4+\xi^2}$ and $"\wedge"$ denotes the Fourier transform on both  time and space.
\end{definition}
\noindent  Moreover, in order to obtain a better cancellation for the bilinear estimates, an equivalence relation in norms,
\[\|u\|_{X^{s,b}(\R^2)}\sim \|\langle \xi\rangle^s\langle |\tau|-|\xi|^3+\frac{\beta}{2} |\xi|\rangle^b\widehat{u}(\xi,\tau)\|_{L^2_{\xi,\tau}(\R^2)},\]
will be adapted which is inspired by  Farah  \cite{34,41}. In addition,  the Bourgain-type space defined here is for functions on the whole plane, $(x,t)\in\R\times \R$. However, the IBVP (\ref{0-1}) is only  posed on a quarter plane, $\R^+\times \R^+$. It is therefore necessary to define a restricted  Bourgain-type space as followed,
\[X^{s,b}\left(\R^+\times \Omega\right):=X^{s,b}|_{\R^+\times \Omega}, \quad \Omega\subset \R^+,\]
with the quotient norm,
\[\|u\|_{X^{s,b}\left(\R^+\times \Omega\right)}:=\inf_{w\in X^{s,b}(\R^2)}\{\|w\|_{X^{s,b}(\R^2)}: w(x,t)=u(x,t)\quad \mbox{on} \quad \R^+\times \Omega\}.\]
 For the sake of simplicity on notation, we consider $X^{s,b}$ as $X^{s,b}(\R^+\times \R^+)$ if no domain are imposed.    The main result for this article is as below.
\begin{theorem}\label{theorem}
Let $-\frac12<s\leq 0$ be given, there exists a $T>0$ such that for any
\[(\varphi,\psi,h_1,h_2,h_3)\in H^s(\R^+)\times H^{s-1}(\R^+)\times  H^{\frac{s+1}{3}}(\R^+)\times H^{\frac{s-1}{3}} (\R^+)\times H^{\frac{s-3}{3}} (\R^+), \]
the IBVP (\ref{0-1}) admits a unique solution
\[u\in C(0,T; H^s(\R^+))\cap X^{s,b}(\R^+\times (0,T))\]
for  $\frac12-b>0$ sufficient small. Moreover, the corresponding solution map is Lipschitz continuous.
\end{theorem}

It is notable that the theorem shows no difference in the local well-posedness for different $\beta$ (consider  $\beta=\pm 1$). However, without the sixth order term (i.e. $u_{xxxxxx}$ in (\ref{0-1})), theories for the ``bad" ($\beta=-1$) and ``good" ($\beta=1$) Boussinesq equations are quite distinct. In addition, the conclusion, $s>-\frac12$, for the half line problem of the sixth order Boussinesq equation is better when comparing to, $s>-\frac14$ (c.f. \cite{97}), the best result for the Boussinesq equation.  Moreover, based on the proof of the theorem, the local well-posedness will not be compromised for other types of boundary conditions, such as
\begin{equation*}
    u(0,t)=h_4(t),u_{x}(0,t)=h_5(t),u_{xx}(0,t)=h_6(t),
\end{equation*}
\noindent or
\begin{equation*}
    u_x(0,t)=h_7(t),u_{xx}(0,t)=h_8(t),u_{xxx}(0,t)=h_9(t).
\end{equation*}

The main idea of this paper is  inspired by the works of Bona-Sun- Zhang \cite{17, bsz-5, bsz-6},  Colliander-Kenig \cite{colliander} and Holmer \cite{holmer, holmer-2}  for the IBVP of the KdV equation \cite{17,18},
 the key  is to explore an explicit integral formula for the associated linear problem:
\begin{equation}\label{01}
\begin{cases}
u_{tt}-u_{xx}+\beta u_{xxxx}-u_{xxxxxx}=0,\quad x>0,t>0,\\
u(x,0)=0, u_t(x,0)=0, \\
u(0,t)=h_1(t), u_{xx}(0,t)=h_2(t), u_{xxxx}(0,t)=h_3(t).
\end{cases}
\end{equation}
However, because of $s\leq0$, the  Strichartz estimates that we established for the half line problem in \cite{96} would not   help, new estimates including  the one in the Bourgain-type space, $X^{s,b}$,  will be introduced.  For $-\frac12<s\leq0$ and $ b $ sufficiently close to $\frac12$, one has
\begin{align*}
    \sup_{t\geq 0} \|u(\cdot,t)\|_{H_x^s(\mathbb{R}^+)} + \sup_{x \geq 0}\|\partial&^j_xu(x,\cdot)\| _{H^{\frac{s-j+1}{3}}(\mathbb{R}^+)}+\|\eta(t)u(x,t)\|_{X^{s,b}}\\
    &\lesssim \|(h_1,h_2,h_3)\|_{H^{\frac{s+1}{3}}(\mathbb{R}^+)\times H^{\frac{s-1}{3}}(\mathbb{R}^+)\times H^{\frac{s-3}{3}}(\mathbb{R}^+)},
\end{align*}
with $j=0,2,4$ and $\eta$ being a cut-off function. Such regularities of the boundary data perfectly match the sharp Kato smoothing
 $$
\sup_{x\in \mathbb{R}}\|\partial_x^j u(x,\cdot)\|_{H^{\frac{s-j+1}{3}}(\mathbb{R}^+)}\lesssim \|(p,q)\|_{H^s(\mathbb{R})\times H^{s-1}(\mathbb{R})}, \quad j=0,2,4,
$$
 for the pure initial-value problem,
 \begin{equation*}
    \begin{cases}
    u_{tt}-u_{xx}+\beta u_{xxxx}-u_{xxxxxx}=0,\quad x\in \mathbb{R},t>0,\\
    u(x,0)=p(x),u_t(x,0)=q''(x).
    \end{cases}
\end{equation*}
 This sharp Kato smoothing provides us a necessary condition for the regularities of the boundary data which also happens to be a sufficient condition for its wellposedness. Combining with the related estimates on its Cauchy problem (c.f. \cite{34}), we are able to study the IBVP,
\begin{equation*}
    \begin{cases}
    u_{tt}-u_{xx}+\beta u_{xxxx}-u_{xxxxxx}=0,\quad x>0,t>0,\\
    u(x,0)=\varphi(x),u_t(x,0)=\psi''(x),\\
    u(0,t)=0,u_{xx}(0,t)=0,u_{xxxx}(0,t)=0,
    \end{cases}
\end{equation*}
and show that for  $-\frac12<s\leq0$ and $ b $ sufficiently close to $\frac12$,
\begin{align*}
    \sup_{ t\geq 0}\|u(\cdot,t)\|_{H^s(\mathbb{R}^+)}+\sup_{x\geq 0}\|\partial_x^ju(x,\cdot)\|_{H^{\frac{s-j+1}{3}}(\R^+)}& +\|\eta(t) u(x,t) \|_{X^{s,b}} \\
   & \lesssim   \|\varphi \|_{H^s(\mathbb{R}^+)}+\|\psi\|_{H^{s-1}(\mathbb{R}^+)},
\end{align*}
 with $j=0,2,4$.
Similar argument can be applied to the forced linear problem,
\begin{equation*}
\begin{cases}
u_{tt}-u_{xx}+\beta u_{xxxx}-u_{xxxxxx}=g_{xx}(x,t),\quad x>0,t>0,\\
u(x,0)=0, u_t(x,0)=0, \\
u(0,t)=0, u_{xx}(0,t)=0, u_{xxxx}(0,t)=0,
\end{cases}
\end{equation*}
 via the Duhamel's principle. However, some of the related estimates for the forced linear problem are only valid for $b<\frac12$. A   proper  bilinear estimate for $b<\frac12$ will be needed, because of the one that established for the Cauchy problem \cite{34} only valids for $b>\frac12$.

The paper is organized as follows:  In section 2, some important notations, basic lemmas of inequalities and  existing results on the Cauchy problem of the sixth order Boussinesq equation are presented. Section 3 is devoted to consider the associated linear problems and it is divided into  two subsections.  Explicit solution  formulas for the associated linear problems will be derived in subsection 3.1.  The sharp Kato smoothing property, $X^{s,b}$ estimates and other related estimates will be provided for the linear problems in subsection 3.2.    Section 4  shows a valid bilinear estimate for $b<\frac12$. The local well-posedness will be established  in Section 5 through contraction mapping principle.

\section{Preliminary}
This section introduces some important notations and lemmas that will be used in the later study. Followings are some important lemmas of elementary calculus inequalities and the reader can refer both of the proof in \cite{95,34}.
\begin{lemma}\label{tz1}
If $\alpha,\beta\geq 0$, $\alpha+\beta>1$ and $\gamma=\min\{\alpha,\beta, \alpha+\beta-1\}$, then,
\[\int_\R \frac{1}{\langle x-a\rangle^\alpha\langle x-b\rangle^\beta}dx\lesssim \frac{1}{\langle a-b\rangle^{-\gamma}}.\]
\end{lemma}

\begin{lemma}\label{tz2}
For $l\in(\frac12,1)$,
\[\int_\R \frac{1}{\langle x\rangle^l \sqrt{|x-a|}}\lesssim \frac{1}{\langle a\rangle^{l-\frac12}}.\]
\end{lemma}
In addition, the Young's inequality will be used multiple times in the proof.
\begin{lemma}\label{young}
For given $f\in L^p(\R)$ and $g\in L^q(\R)$  with
\[\frac{1}{p}+\frac{1}{q}=1+\frac{1}{r}, \quad 1\leq p,q,r \leq \infty,\]
then
\[\|f*g\|_{L^r(\R)}\leq \|f\|_{L^p(\R)}\|g\|_{L^q(\R)}.\]
\end{lemma}

Next, the  lemma below comes from \cite{34}.
\begin{lemma}\label{equ}
There exists a $c>0$ such that
$$
\frac{1}{c}\leq \sup_{x,y\geq 0} \frac{1+|x-\sqrt{y^3}-\frac{\beta}{2}\sqrt{y}|}{1+|x-\sqrt{y^3+\beta y^2+y}|}\leq c,\ \ \ \beta=\pm1.
$$
\end{lemma}
 This  can lead to a equivalence between the norm of the Bourgain-type space for the sixth order Bousssinesq and the one for the KdV-type equation (c.f. \cite{54,53,52,50}), that is,
\[\|u\|_{X^{s,b}(\R^2)}\sim \|\langle \xi\rangle^s\langle |\tau|-|\xi|^3-\frac{\beta}{2} |\xi|\rangle^b\widehat{u}(\xi,\tau)\|_{L^2_{\xi,\tau}(\R^2)}.\]

For convenience,   $\eta$ denotes a cut-off function all through this article satisfying $\eta\in C^\infty (\R)$ with
\[\eta=\begin{cases}
1,\quad x\in [-1,1],\\
0, \quad x<-2 \mbox{ or }x>2.
\end{cases}\]
In addition, $\eta_T(t):=\eta(t/T)$ for $0<T<1$. The following lemma is introduced in \cite{95} and will be useful in the later proof.

\begin{lemma}\label{a}
For any $-\frac12<b'<b<\frac12$, $\|\eta_T(t)u\|_{X^{s,b}}\lesssim T^{b'-b}\|u\|_{X^{s,b'}}$.
\end{lemma}

Finally, some of the important  notations and previous results on the Cauchy problem of the sixth order Boussinesq equation are provided below. Let us denote  $u=[W_\R(f_1,f_2)](x,t)$ as the solution to the linear problem,
\begin{equation}\label{linear1}
\begin{cases}
u_{tt}-u_{xx}+\beta u_{xxxx}-u_{xxxxxx}=0, \ \ \ x\in\R, t>0,\\
u(x,0)=f_1(x),u_t(x,0)=f_2''(x).
\end{cases}
\end{equation}
and write
$
    [W_{\R}(f_1,f_2)](x,t):=[V_1(f_1)](x,t)+[V_2(f_2)](x,t)
$
with
\begin{equation}\label{v01}
   [ V_1(f_1)](x,t):=\frac{1}{2}\int_{\R} \left(e^{i(t\phi(\xi)+x\xi)}+e^{{i(-t\phi(\xi)+x\xi)}}\right)\widehat{f}_1(\xi)d\xi,
\end{equation}
and
\begin{equation}\label{v02}
   [ V_2(f_2)](x,t):=\frac{1}{2i}\int_{\R } \left(e^{i(t\phi(\xi)+x\xi)}-e^{i(-t\phi(\xi)+x\xi)}\right)\frac{\xi^2\widehat{f}_2(\xi)}{\phi(\xi)}d\xi,
\end{equation}
The estimates below  comes from \cite{34,35,96}.
\begin{lemma}\label{R}
Given $f_1 \in H^s(\R)$ and $f_2\in H^{s-1}(\R)$, for any $s$ and $b$, the solution $u$ of the IVP (\ref{linear1}) satisfies
\begin{align*}
\sup_{t\geq 0} \|u(\cdot,t)\|_{H^s(\R)}&\lesssim \|f_1\|_{H^s(\R)}+\|f_2\|_{H^{s-1}(\R)},\\
 \sup_{x\geq 0} \| \partial ^j_x u(x,\cdot)\|_{H^{\frac{s-j+1}{3}}(\mathbb{R})}&\lesssim\|f_1 \|_{H^s(\mathbb{R})}+\|f_2 \|_{H^{s-1}(\mathbb{R})},\\
   \|\eta(t)u(x,t)\|_{X^{s,b}(\R^2)}&\lesssim \|f_1\|_{H^s(\R)}+\|f_2\|_{H^{s-1}(\R)},
\end{align*}
for $j=0,1,2,..,5$.
\end{lemma}
According to the Duhamel's principal, the  IVP for the forced linear equation,
 \begin{equation*}
\begin{cases}
u_{tt}-u_{xx}+\beta u_{xxxx}-u_{xxxxxx}=g_{xx}(x,t),\quad x\in \mathbb{R}\mbox{, }t>0,\\
u(x,0)=0, u_t(x,0)=0,
\end{cases}
\end{equation*}
has its solution $u$ in the form
\[u(x,t)=\int^t_0 [W_\R(0,g)](x,t-\tau)d\tau.\]
The following lemma comes from \cite{34}.
\begin{lemma}\label{f1}
Let $-\frac12< b'\leq 0 \leq b\leq b'+1$, then for any $s$,
\begin{align*}
 \left\|\eta_T(t)\int^t_0 [W_\R(0,g)](x,t-\tau)d\tau\right\|_{X^{s,b}(\R^2)}\lesssim T^{1-b+b'}\left\|\left(\frac{\xi^2\widehat{g}(\xi,\tau)}{2i\phi(\xi)}\right)^\vee\right\|_{X^{s,b'}(\R^2)},
\end{align*}
where $\phi(\xi)=\sqrt{\xi^6+\beta \xi^4+\xi^2}$, `` $\wedge $" and `` $\vee$" denote the  Fourier transform  and inverse Fourier transform in both time and space.
\end{lemma}
\section{Linear Problems}
This section concerns on  looking for an explicit formula  for the IBVP of the associated forced linear equation:
\begin{equation}\label{lin-ex}
\begin{cases}
u_{tt}-u_{xx}+\beta u_{xxxx}-u_{xxxxxx} =g_{xx}(x,t),\quad x>0\mbox{, }t>0,\\
u(x,0)=\varphi (x), u_t(x,0)=\psi ''(x), \\
u(0,t)=h_1(t), u_{xx}(0,t)=h_2(t), u_{xxxx}(0,t)=h_3(t),
\end{cases}
\end{equation}
along with some related estimates.
\subsection{ Boundary integral operators and solution formulas}

 First, we consider the linear problem with homogeneous initial conditions,
\begin{equation}\label{homoin}
\begin{cases}
u_{tt}-u_{xx}+\beta u_{xxxx}-u_{xxxxxx} =0,\quad x>0\mbox{, }t>0,\\
u(x,0)=0, u_t(x,0)=0, \\
u(0,t)=h_1(t), u_{xx}(0,t)=h_2(t), u_{xxxx}(0,t)=h_3(t).
\end{cases}
\end{equation}
Applying the Laplace transform with respect to $t$,   the IBVP  (\ref{homoin})    is converted to
\begin{equation}\label{eq1}
\begin{cases}
\rho^2\widetilde{u}-\widetilde{u}_{xx}+ \beta \widetilde{u}_{xxxx}-\widetilde{u}_{xxxxxx}=0,\\
\widetilde{u}(0,\rho)=\widetilde{h}_1(\rho), \widetilde{u}_{xx}(0,\rho)=\widetilde{h}_2(\rho), \widetilde{u}_{xxxx}(0,\rho)=\widetilde{h}_3(\rho),\\
\widetilde{u}(+\infty,\rho)=0, \widetilde{u}_{xx}(+\infty,\rho)=0, \widetilde{u}_{xxxx}(+\infty,\rho)=0,
\end{cases}
\end{equation}
where
\begin{equation*}
    \widetilde{u}(x,\rho)=\int^{+\infty}_0 e^{-\rho t}u(x,t)dt,
 \qquad
   \widetilde{h}_j(\rho)=\int^{+\infty}_0 e^{-\rho t}h_j(t)dt, \hspace{0.5cm}j=1,2,3.
\end{equation*}
For any $\rho$ with $Re (\rho) > 0$ satisfying $\rho\ne  \frac{2\sqrt{2}\pm i}{3\sqrt{3}}$, the solution $\widetilde{u}(x,\rho)$ of  (\ref{eq1}) can be written  in the form
\begin{equation}\label{lap}
    \widetilde{u}(x,\rho)=\sum^3_{j=1}c_{j}(\rho)e^{\gamma_{j}(\rho)x},
\end{equation}
where  $\gamma _1=\gamma _1 (\rho) $, $\gamma _2=\gamma _2 (\rho) $ and $\gamma _3=\gamma _3 (\rho) $,   $Re ( \gamma _j) <0$  for $\ j=1,2,3$, are  three solutions  of the characteristic equation
\begin{equation}\label{ch}
    \gamma^6-\gamma^4+\gamma^2-\rho^2=0,
\end{equation}
and  $c_j(\rho), \mbox{ for }j=1,2,3$,  solve the linear system
\begin{equation}\label{system}
\begin{cases}
c_{1}+c_{2}+c_{3}=\widetilde{h}_1(\rho),\\
\gamma_{1}^2c_{1}+\gamma_{2}^2c_{2}+\gamma_{3}^2c_{3}=\widetilde{h}_2(\rho),\\
\gamma_{1}^4c_{1}+\gamma_{2}^4c_{2}+\gamma_{3}^4c_{3}=\widetilde{h}_3(\rho).
\end{cases}
\end{equation}

Let $\Delta(\rho)$ be the determinant of the coefficient matrix of \eqref{system} and $\Delta_{j}(\rho)$ be the determinants of the matrices with the $j-$th column replacing by $(\widetilde{h}_1(\rho),\widetilde{h}_2(\rho),\widetilde{h}_3(\rho))^T$ for $j=1,2,3$. In addition, for $j,m=1,2,3$,  $\Delta_{j,m}(\rho)$ is  obtained from $\Delta_j(\rho)$ by letting $\widetilde{h}_m(\rho)=1$ and $\widetilde{h}_j(\rho)=0$ for $j\neq m$. Moreover,  $\gamma_j^+(\mu)$  for $j=1,2,3$ denote the three distinct roots of the characteristic equation (\ref{ch}) by changing the variable $\rho=i \phi(\mu)$ where
\begin{equation}\label{mu0}
\phi(\mu)=\mu \sqrt{\mu^4+\mu^2+1}, \quad \mbox{for }\quad \mu>0.
\end{equation}
Thus, one can solve,
\begin{equation}\label{gamma}
  \gamma_1^+ (\mu) = i\mu, \quad  \gamma_2^+ (\mu) = -p(\mu)-iq(\mu), \quad  \gamma_3^+ (\mu) = -p(\mu)+iq(\mu),
\end{equation}
where
\[p(\mu)=\frac{1}{\sqrt{2}}\left(\sqrt{\mu^2+1+\sqrt{4\mu^4+4\mu^2+4}}\right),\]
\[q(\mu)=\frac{1}{\sqrt{2}}\left(\sqrt{\sqrt{4\mu^4+4\mu^2+4}-\mu^2-1}\right).\]
Taking inverse Laplace transform of \eqref{lap}, one  can obtain the formula for the solution of \eqref{homoin} presented in the following lemma.
\begin{lemma}\label{bdr}
Given $\vec{h} = (h_1,h_2,h_3)$, the solution $u$ of the IBVP (\ref{homoin}) can be written in the form
\[ u(x,t) =[W_{bdr} (\vec{h})] (x,t):= \sum _{m=1}^3 [W_{bdr,m} (h_m)](x,t)  \]
with
\[ [W_{bdr,m}(h_m) ](x,t)= I_m (x,t) +\overline{I_m(x,t)}, \qquad m=1,2,3,\]
where
\begin{equation*}
    I_m(x,t)=\frac{1}{2\pi }\sum^3_{j=1}\int^{+\infty}_0 e^{i\tau\mu \sqrt{\mu^4+\mu^2+1}t}e^{\gamma_{j}^+(\mu)x}
    \frac{\Delta_{j,m}^+(\mu)}{\Delta^+(\mu)}\frac{3\mu^4+2\mu^2+1}{\sqrt{\mu^4+\mu^2+1}}\widetilde{h}_m^+(\mu)d\mu,
\end{equation*}
and
\[\widetilde{h}_m^+(\mu)=\int^\infty_0 e^{i\phi(\mu )t}h_m(t)dt.\]
\end{lemma}

Next, for the linear problems with homogeneous boundary conditions,
\begin{equation}\label{in2-0}
    \begin{cases}
    u_{tt}-u_{xx}+ \beta u_{xxxx}-u_{xxxxxx}=0,\quad x>0 \mbox{, }t>0,\\
    u(x,0)=\varphi (x), u_t(x,0)=\psi ''(x),\\
    u(0,t)=0, u_ {xx}(0,t)=0, u_{xxxx}(0,t)=0,
    \end{cases}
\end{equation}
and
\begin{equation}\label{in2-1}
    \begin{cases}
    u_{tt}-u_{xx}+ \beta u_{xxxx}-u_{xxxxxx}=g_{xx}(x,t),\quad x>0 \mbox{, }t>0,\\
    u(x,0)=0, u_t(x,0)=0,\\
    u(0,t)=0, u_ {xx}(0,t)=0, u_{xxxx}(0,t)=0,
    \end{cases}
\end{equation}
we have  the following two lemmas based on Lemma \ref{bdr}.
\begin{lemma}\label{nonf}  For any $s$, given  $\varphi\in H^s (\R ^+)$ and  $\psi \in H^{s-1} (\R ^+)$, let   $\varphi ^*, \ \psi ^*$   be  extensions of $\varphi $ and $\psi $ from $\R^+\to \R$,  respectively.
For \[p(x,t)=[ W_{\R} (\varphi ^*, \psi ^*)] (x,t),\]  set $\vec{p}=(p_1,p_2,p_3)$ and
\[ p_1 (t)=p(0,t), \quad p_2 (t)= p_{xx} (0,t), \quad p_3 (t)=p_{xxxx} (0,t). \]
Then the solution $u $ of the IBVP (\ref{in2-0}) can be written in the form
\[ u(x,t) = [W_{\R} (\varphi ^*, \psi ^*)] (x,t)- [W_{bdr} (\vec{p})] (x,t) .\]
\end{lemma}

\begin{lemma}\label{force} For $g\in C^\infty([0,T]\times\R^+)$, let $g^*\in C^\infty([0,T]\times\R)$ be   extension of $g$ from $\R^+ $ to $\R$. Then the corresponding solution $u$ of (\ref{in2-1}) can be written as
\[ u(x,t) = \int ^t_0 [  W_\mathbb{R} (0,g^*)](x, t-\tau ) d\tau - [W_{bdr} (\vec{q})](x,t)\]
where $\vec{q}=(q_1,q_2,q_3)$ and
\[ q_1 (t)=q(0,t), \quad q_2 (t)= q_{xx} (0,t), \quad q_3 (t)=q_{xxxx} (0,t) \]
with
\[ q(x,t) = \int ^t_0 [  W_\mathbb{R} (0,g^*)](x, t-\tau ) d\tau .\]
\end{lemma}
 For more details on the generation of these explicit  solutions of the IBVPs, the reader can refer to \cite{96}.
\subsection{Linear estimate}
In this subsection,   the estimates of the solutions described in Lemma \ref{bdr}, \ref{nonf} and \ref{force} are considered.
\subsubsection{Linear problem with homogeneous initial conditions}
First, let us   estimate  the solution of the IBVP
\begin{equation} \label{ln-1}
\begin{cases}
u_{tt}-u_{xx}+\beta u_{xxxx}-u_{xxxxxx}=0,\quad x>0,t>0,\\
u(x,0)=0, u_t(x,0)=0, \\
u(0,t)=h_1(t), u_{xx}(0,t)=h_2(t), u_{xxxx}(0,t)=h_3(t).
\end{cases}
\end{equation}
We denote $\vec{h}=(h_1,h_2,h_3)$ and provide the following conclusions on $W_{bdr}$ introduced in Lemma \ref{bdr}.
\begin{proposition}\label{1}
For $s\leq 0$,
\[\sup_{t\geq0}\|[W_{bdr}(\vec{h})](\cdot,t)\|_{H^{s}(\R^+)}\lesssim \|h_1\|_{H^{\frac{s+1}{3}}(\R^+)}+\|h_2\|_{H^{\frac{s-1}{3}}(\R^+)}+\|h_3\|_{H^{\frac{s-3}{3}}(\R^+)},\]
\[\sup_{x\geq0}\|\partial_x^k[ W_{bdr}(\vec{h})](x,\cdot)\|_{H^{\frac{s-k+1}{3}}(\R^+)}\lesssim \|h_1\|_{H^{\frac{s+1}{3}}(\R^+)}+\|h_2\|_{H^{\frac{s-1}{3}}(\R^+)}+\|h_3\|_{H^{\frac{s-3}{3}}(\R^+)},\]
with $k=0,1,2,...,5.$
\end{proposition}

\begin{proposition}\label{2}
For $s\leq 0$ and $\left|\frac12-b\right|$ sufficient small,
\[\|\eta(t)W_{bdr}(\vec{h})\|_{X^{s,b}}\lesssim \|h_1\|_{H^{\frac{s+1}{3}}(\R^+)}+\|h_2\|_{H^{\frac{s-1}{3}}(\R^+)}+\|h_3\|_{H^{\frac{s-3}{3}}(\R^+)}.\]
\end{proposition}
\begin{proof}
(\textbf{Proposition \ref{1}}) The proof for the second estimate is same as for $s\geq0$ in \cite{96}. For the first estimate, it suffices to  prove for $\vec{h}=(h_1,0,0)$. According to Lemma \ref{bdr},
 \[[W_{bdr}(h_1,0,0)](x,t)=I_1(x,t)+\overline{I_1(x,t)},\]
where
\begin{align*}
    \emph{\textrm{I}}_1(x,t)=&\frac{1}{2\pi }\sum^3_{j=1}\int^{+\infty}_{0} e^{i\phi(\mu)t}e^{\gamma^+_{j}(\mu)x}
    \frac{\Delta^+_{j,m}(\mu)}{\Delta^+(\mu)}\frac{3\mu^4+2\mu^2+1}{\sqrt{\mu^4+\mu^2+1}}\widetilde{h}^+_1(\mu)d\mu\\
    :=&\frac{1}{2\pi }\sum^3_{j=1}I_{1,j},
\end{align*}
with
$\widetilde{h}_1^+(\mu)=\int^\infty_0 e^{i\phi(\mu )t}h_1(t)dt$ and $\gamma_1^+$, $\gamma_2^+$ ,$\gamma_3^+$ given in \eqref{gamma}.

 It suffices to show the estimate for $I_1$. We first estimate $I_{1,1}$, and note that $\gamma_1^+ (\mu) = i\mu$, through direct computation, one has
\[\frac{\Delta^+_{1,1}(\mu)}{\Delta^+(\mu)} \sim 1 \quad  \mbox{and }\quad \frac{3\mu^4+2\mu^2+1}{\sqrt{\mu^4+\mu^2+1}}\sim \mu^2\quad \mbox{as}\ \ \mu\rightarrow  \infty,\]
thus,
\begin{align*}
\|I_{1,1}(\cdot,t)\|_{H^s(\R^+)}\lesssim &\left\|\int^{+\infty}_{0} e^{i\mu x}
    \mu^2\widetilde{h}^+_1(\mu)d\mu\right\|_{H^s(\R^+)}\\
    \lesssim &\left\|\mu^s \mu^2 \widetilde{h}^+_1(\mu)\right\|_{L^2_\mu(\R^+)}\\
    \lesssim & \left\|\zeta^{\frac{n+1}{3}}\int^{+\infty}_0 e^{i\zeta \tau}h_1(\tau)d\tau\right\|_{L^2_\zeta(\R^+)}\lesssim \|h_1\|_{H^{\frac{s+1}{3}}(\R^+)}.
\end{align*}
Next, we consider the estimate for $I_{1,2}$, and note that $\gamma_2^+ (\mu) = -p(\mu)-iq(\mu)$, through direct computation, one has
\[\frac{\Delta^+_{2,1}(\mu)}{\Delta^+(\mu)} \sim 1 \quad  \mbox{and }\quad \frac{3\mu^4+2\mu^2+1}{\sqrt{\mu^4+\mu^2+1}}\sim \mu^2\quad \mbox{as}\ \ \mu\rightarrow  \infty,\]
thus,
\begin{align*}
\|I_{1,2}(\cdot,t)\|_{H^s(\R^+)}\lesssim &\left\|\int^{+\infty}_{0} e^{-iq(\mu) x}
    \mu^2\widetilde{h}^+_1(\mu)d\mu\right\|_{H^s(\R^+)}\\
    \lesssim &\left\|\int^{+\infty}_{0} e^{-iq x}
     \Big(\mu(q)\Big)^2\frac{d\mu}{dq}\widetilde{h}^+_1(q)dq\right\|_{H^s(\R^+)}\\
    \lesssim & \left\|q^s \Big(\mu(q)\Big)^2\frac{d\mu}{dq}\widetilde{h}^+_1(q)\right\|_{L^2_{q}(\R^+)}\\
    \lesssim& \left\| \Big(q(\mu)\Big)^s \mu^2(\frac{d\mu}{dq})^{\frac12}\widetilde{h}^+_1(\mu)\right\|_{L^2_{\mu}(\R^+)}  \lesssim  \|h_1\|_{H^{\frac{s+1}{3}}(\R^+)},
\end{align*}
since $q(\mu)\sim \mu$ as $\mu\rightarrow \infty$. The estimate on $I_{1,2}$ can be obtained similar to $I_{1,2}$. The proof is now complete.
\end{proof}

\begin{proof}
(\textbf{Proposition \ref{2}}) Similar  to Proposition \ref{1}, it suffices to prove for $\vec{h}=(h_1,0,0)$. Let us write
 \[[W_{bdr}(h_1,0,0)](x,t)=I_1(x,t)+\overline{I_1(x,t)},\]
with
\begin{align*}
    \emph{\textrm{I}}_1(x,t)&=\frac{1}{2\pi }\sum^3_{j=1}\int^{+\infty}_{0} e^{i\phi(\mu)t}e^{\gamma^+_{j}(\mu)x}
    \frac{\Delta^+_{j,m}(\mu)}{\Delta^+(\mu)}\frac{3\mu^4+2\mu^2+1}{\sqrt{\mu^4+\mu^2+1}}\widetilde{h}^+_1(\mu)d\mu\\
    &:=I_{1,1}+I_{1,2}+I_{1,3}.
\end{align*}
It  suffices to show the estimate for $I_1$.

For $j=1$, one has $\gamma_1^+ (\mu) = i\mu$,  according to Lemma \ref{R},
\begin{align*}
\|\eta I_{1,1}\|_{X^{s,b}}&\lesssim \|\eta W_{\R}(\Phi,0)\|_{X^{s,b}(\R^2)}\lesssim \|\Phi\|_{H^s(\R)}\lesssim \|h_1\|_{H^{\frac{s+1}{3}}(\R^+)},
\end{align*}
where $W_\R$ is defined in \eqref{v01} and \eqref{v02} with the Fourier transform function of $\Phi(x)$ defined by:
\begin{equation*}
\widehat{\Phi}(\mu)=\frac{\Delta_{j,m}(\mu)}{\Delta^+(\mu)}\frac{3\mu^4+2\mu^2+1}{\sqrt{\mu^4+\mu^2+1}}\widetilde{h}^+_1(\mu)
\chi_{(0,\infty)}(\mu).
\end{equation*}

For $j=2$, one has $\gamma_2^+ (\mu)=-p(\mu)-iq(\mu)$.
Let $f(x)=e^{-x}\theta(x)$ where $\theta$ is a cut-off function such that $\theta(x)=1$ for $x\geq0$ and $\theta(x)=0$ for $x<-1$.
 Thus, for $(x,t)\in \R^+\times \R^+$, let us write
\begin{align*}
 \eta(t)I_{1,2}(x,t)=\int_\R \Big(\eta(t)e^{i\phi(\mu)t}\Big)\Big(f\left(xp(x)\right)e^{-iq(\mu)x}\Big)
    &\frac{\Delta^+_{j,m}(\mu)}{\Delta^+(\mu)}\\
    &\cdot\frac{3\mu^4+2\mu^2+1}{\sqrt{\mu^4+\mu^2+1}}\widetilde{h}^+_1(\mu)d\mu.
    \end{align*}
 Since $\eta$ and $\theta$ are cut-off functions defined on $\R$, one can extend $\eta(t)I_{1,2}(x,t)$ to the entire plane, $\R\times \R$, such that
 \[ \widehat{\eta I}_{1,2}(\xi,\tau)  =\int_\R \widehat{\eta}(\tau-\phi(\mu)) \widehat{f}\left(\frac{\xi+q}{p}\right)\frac{1}{p}\frac{\Delta^+_{j,m}(\mu)}{\Delta^+(\mu)}\frac{3\mu^4+2\mu^2+1}{\sqrt{\mu^4+\mu^2+1}}\widetilde{h}^+_1(\mu)d\mu.
\]
Notice that $f$ is a Schwartz function with $$p=\sqrt{\sqrt{\mu^4+\mu^2+1}+\frac12\mu^2+\frac12}, \quad q=\sqrt{\sqrt{\mu^4+\mu^2+1}-\frac12\mu^2-\frac12},$$  one has
\begin{equation*}
\left|\widehat{f}\left(\frac{\xi+q}{p}\right)\right|\lesssim \frac{1}{1+(\xi+q)^3/p^3} \lesssim \frac{1+|\mu|^3}{1+ |\xi|^3+ k|\mu|^3},
\end{equation*}
for some $k>1$. Since $p\sim |\mu|$ as $\mu\rightarrow +\infty$ and $\langle\tau-\phi(\xi)\rangle\lesssim \langle\tau-\phi(\mu)\rangle\langle\phi(\mu)-\phi(\xi)\rangle$, applying Young's inequality (c.f. Lemma \ref{young}) follows
\begin{align*}
&\|\eta I_{1,2}\|_{X^{s,b}(\R^2)}\\
=&\left\|\langle \xi\rangle^s\langle \tau-\phi(\xi)\rangle^{b} \int_\R \widehat{\eta}(\tau-\phi(\mu)) \widehat{f}\left(\frac{\xi+q}{p}\right)\frac{1}{p}\frac{\Delta^+_{j,m}(\mu)}{\Delta^+(\mu)}\frac{3\mu^4+2\mu^2+1}{\sqrt{\mu^4+\mu^2+1}}\widetilde{h}^+_1(\mu)d\mu \right\|_{L^2_\xi L^2_\tau}\\
\lesssim& \left\|\int_\R \langle\tau-\phi(\mu)\rangle^b\widehat{\eta}(\tau-\phi(\mu))\frac{\langle\phi(\mu)-\phi(\xi)\rangle^b\langle \xi\rangle^s(1+|\mu|^3)}{1+ |\xi|^3+ k|\mu|^3}|\mu| \widetilde{h}^+_1(\mu)d\mu \right\|_{L^2_\xi L^2_\tau}\\
\lesssim& \left\|\int_\R \langle\tau-\phi(\mu)\rangle^b\widehat{\eta}(\tau-\phi(\mu))\left\|\frac{\langle\phi(\mu)-\phi(\xi)\rangle^b\langle \xi\rangle^s}{1+ |\xi|^3+ k|\mu|^3}\right\|_{L^2_\xi}(1+|\mu|^3)|\mu| \widetilde{h}^+_1(\mu)d\mu\right\|_{ L^2_\tau}\\
\lesssim& \left\|\int_\R \langle\tau-\phi(\mu)\rangle^b\widehat{\eta}(\tau-\phi(\mu)) |\mu|^{s-1}|\mu|^4\widetilde{h}^+_1(\mu)d\mu\right\|_{ L^2_\tau}\\
\lesssim& \left\|\int_\R \langle\tau-\zeta\rangle^b\widehat{\eta}(\tau-\zeta) |\zeta|^{\frac{s+1}{3}}\int^\infty_0 e^{i\zeta r}h_1(r)dr d\zeta\right\|_{ L^2_\tau}\lesssim \|h_1\|_{H^{\frac{s+1}{3}}(\R^+)}.
\end{align*}
The above bound holds because
\begin{align*}
\int_\R \frac{\langle\phi(\mu)-\phi(\xi)\rangle^{2b}\langle \xi\rangle^{2s}}{(1+ |\xi|^3+ k|\mu|^3)^2}d\xi&\lesssim \int_\R \frac{\langle |\mu|^3 +|\xi|^3\rangle^{2b} \langle \xi\rangle^{2s}}{(1+ |\xi|^3+ k|\mu|^3)^2}d\xi\\
&\lesssim \int_\R \frac{ \langle \xi\rangle^{2s}}{(1+ |\xi|^3+ k|\mu|^3)^{2-2b}}d\xi,
\end{align*}
and
\begin{align*}
&\int_{|\xi|\lesssim |\mu|} \frac{ \langle \xi\rangle^{2s}}{(1+ |\xi|^3+ k|\mu|^3)^{2-2b}}d\xi\lesssim |\mu|^{-3(2-2b)}\int_{|\xi|\lesssim |\mu|}\langle \xi\rangle^{2s}d\xi \lesssim |\mu|^{2s+6b-5},\\
&\int_{ |\xi|\gg  |\mu|} \frac{ \langle \xi\rangle^{2s}}{(1+ |\xi|^3+ k|\mu|^3)^{2-2b}}d\xi\lesssim \int_{ |\xi|\gg  |\mu|}| \xi| ^{2s+6b-6}d\xi \lesssim |\mu|^{2s+6b-5},
\end{align*}
for $s<1$ and $|\frac12-b|$ sufficient small. The proof for $j=3$ is similar to $j=2$, therefore omitted.
The proof is now complete.
\end{proof}
\subsubsection{Linear problem with homogeneous boundary conditions}
Next, let us estimate   the solution of the IBVP,
\begin{equation} \label{ln-2}
\begin{cases}
u_{tt}-u_{xx}+\beta u_{xxxx}-u_{xxxxxx}=0, \quad x>0,t>0,\\
u(x,0)=\varphi (x), u_t(x,0)=\psi ''(x), \\
u(0,t)=0, u_{xx}(0,t)=0, u_{xxxx}(0,t)=0.
\end{cases}
\end{equation}
According to Lemma \ref{nonf}, the solution of IBVP (\ref{ln-2}) can be written as
\[u(x,t)=[W_\R(\varphi^*,\psi^*)](x,t)-[W_{bdr}(\vec{p})](x,t),\]
where $\vec{p}(t)= (p_1(t), p_2(t), p_3 (t))$  with
\[ p_1 (t)= [W_\R(\varphi^*,\psi^*)](0,t), \quad p_2 (t)= \partial_{x}^2[W_\R(\varphi^*,\psi^*)](0,t), \]
and
\[ p_3 (t)=\partial_x^4[W_\R(\varphi^*,\psi^*)](0,t).\]
Hence, combining Propositions \ref{1}, \ref{2} and Lemmas \ref{R}, \ref{nonf}, one has the following statement.
\begin{proposition}\label{hi}
Given $s\leq 0$ and $\left|\frac12-b\right|$ sufficient small, for any $\varphi \in H^s (\R^+)$ and $\psi \in H^{s-1}(\R^+)$, the  corresponding solution $u$  of the IBVP (\ref{ln-2})
satisfies
\begin{align*}
     \sup_{t\geq0}\|u(\cdot,t)\|_{H^s(\R^+)}+ \sup_{x\geq0} \| \partial ^j_x u(x,\cdot)\|_{H^{\frac{s-j+1}{3}}(\mathbb{R}^+)}+&\|\eta(t) u(x,t)\|_{X^{s,b}}\\
     &\lesssim \|\varphi \|_{H^s(\mathbb{R}^+)}+\|\psi\|_{H^{s-1}(\mathbb{R}^+)},
\end{align*}
with $j=0,1,2,..,5$.
\end{proposition}

\subsubsection{Linear problem with forcing}
 Finally, let us estimate  the solution of the IBVP,
\begin{equation}\label{ln-3}
    \begin{cases}
    u_{tt}-u_{xx}+\beta u_{xxxx}-u_{xxxxxx}=g_{xx} (x,t),\quad x>0\mbox{, }t>0,\\
    u(x,0)=0, u_t(x,0)=0,\\
    u(0,t)=0, u_ {xx}(0,t)=0, u_{xxxx}(0,t)=0.
    \end{cases}
\end{equation}

To study the estimate of  IBVP \eqref{ln-3}, according to Lemma \ref{f1}, \ref{force} and Proposition \ref{1}, \ref{2}, \ref{hi}, it is necessary to establish the sharp Kato smoothing of the solution for the forced linear IVP.
\begin{equation}
\begin{cases}
u_{tt}-u_{xx}+\beta u_{xxxx}-u_{xxxxxx}=g_{xx}(x,t),\quad x\in \mathbb{R}\mbox{, }t>0,\\
u(x,0)=0, u_t(x,0)=0,
\end{cases}
\end{equation}

\begin{proposition}\label{k1}
 For $s\leq 0$ and $\frac12-b>0$ sufficient small,
\begin{align*}
&\sup_{x\geq0}\left\|\eta(t)\int^t_0 \partial_x^j[ W_\R(0,g)](x,t-\tau)d\tau\right\|_{H_t^{\frac{s-j+1}{3}}(\R)}\\
\lesssim & \begin{cases}
\left\|\left(\frac{\xi^2\widehat{g}}{\phi(\xi)}\right)^\vee\right\|_{X^{s,-b}},& j=0,3,\\
\left\|\left(\frac{\xi^2\widehat{g}}{\phi(\xi)}\right)^\vee\right\|_{X^{s,-b}}+\left\|\langle
\tau\rangle^{\frac{s}{3}}\int_D\langle\xi\rangle^{-2}\left|\frac{\xi^2\widehat{g}}{\phi(\xi)}\right|d\xi\right\|_{L^2_\tau},& j=1,4,\\
\left\|\left(\frac{\xi^2\widehat{g}}{\phi(\xi)}\right)^\vee\right\|_{X^{s,-b}}+\left\|\langle\tau
\rangle^{\frac{s-1}{3}}\int_D\langle\xi\rangle^{-1}\left|\frac{\widehat{g}}{\phi(\xi)}\right|d\xi\right\|_{L^2_\tau},& j=2,5,
\end{cases}
\end{align*}
where $D$ is the region $\{\xi\in \R:|\xi|^3\gg |\tau|, |\xi|\gtrsim 1\}$.
  \end{proposition}
  \begin{remark}
  We notice that,  unlike in Proposition \ref{2} and \ref{hi}, it is necessary to have the condition $b<\frac12$ for the estimates.
  \end{remark}

  \begin{proof}
Let us start by verify the estimate for $j=0$. It suffices to find the bound at $x=0$ since $X^{s,b}$ is independent of space translation. According to the definition of $W_{\R} $ in \eqref{v01} and \eqref{v02}, one can write,
  \begin{align*}
  \int^t_0 W_\R(0,g)(x,t-\tau)d\tau\Big|_{x=0}=&\int^t_0\int_\R \frac{e^{i(t-\tau)\phi(\xi)}-e^{-i(t-\tau)\phi(\xi)}}{2i\phi(\xi)}\xi^2\widehat{g}^x(\xi,\tau)d\xi d\tau\\
  =&\int^t_0\int_\R  \frac{\xi^2\left(e^{i(t-\tau)\phi(\xi)}-e^{-i(t-\tau)\phi(\xi)}\right)}{2i\phi(\xi)}\\
  &\cdot\left(\int_\R e^{i\lambda \tau}\widehat{g}(\xi,\lambda)d\lambda\right) d\xi d\tau\\
  :=&\frac{1}{2i}(V_1+V_2),
   \end{align*}
  where `` $\wedge_x$'' denotes the Fourier transform with respect to space and
  \begin{align*}
V_1&=\int_{\R^2}\frac{\xi^2e^{it\phi(\xi)}\widehat{g}(\xi,\lambda)}{\phi(\xi)}\left(\int^t_0 e^{i\tau (\lambda-\phi(\xi))}d\tau\right)d\lambda d\xi\\
&=-i\int_{\R^2}\frac{e^{it\lambda}-e^{it\phi(\xi)}}{\lambda-\phi(\xi)}F(\xi,\lambda)d\lambda d\xi,
\end{align*}
\begin{equation*}
V_2=-i\int_{\R^2}\frac{e^{it\lambda}-e^{-it\phi(\xi)}}{\lambda+\phi(\xi)}F(\xi,\lambda)d\lambda d\xi,
\end{equation*}
with
\[F(\xi,\lambda):=\frac{\xi^2\widehat{g}(\xi,\lambda)}{\phi(\xi)}.\]
It suffices to consider the estimate for $V_1$. Let $\theta$ be a cut-off function with $\theta=1$ on $[-1,1]$ and $supp (\theta )\subset (-2,2)$, and  denote $\theta^c=1-\theta$. Then,
\begin{align*}
V_1=&-i\int_{\R^2}\frac{e^{it\lambda}-e^{it\phi(\xi)}}{\lambda-\phi(\xi)}\theta(\lambda-\phi(\xi))F(\xi,\lambda)d\lambda d\xi-i\int_{\R^2}\frac{e^{it\lambda}}{\lambda-\phi(\xi)}\theta^c(\lambda-\phi(\xi))\\
&\cdot F(\xi,\lambda)d\lambda d\xi+i\int_{\R^2}\frac{e^{it\phi(\xi)}}{\lambda-\phi(\xi)}\theta^c(\lambda-\phi(\xi))F(\xi,\lambda)d\lambda d\xi\\
:=&-i(A_1+A_2-A_3).
\end{align*}
For $A_1$, using Taylor expansion, one has
\begin{equation*}
A_1=\int_{\R^2}-e^{i\lambda t}\sum^{\infty}_{k=0}\frac{(it)^{k}(\lambda-\phi(\xi))^{k-1}}{k!}\theta(\lambda-\phi(\xi))F(\xi,\lambda)d\lambda d\xi.
\end{equation*}
Hence,
\begin{align*}
&\|\eta(t)A_1\|_{H_t^{\frac{s+1}{3}}(\R)}\\
\lesssim& \sum^\infty_{k=0}\frac{\|t^k \eta(t)\|_{H^1(\R)}}{k!}\left\|\int_{\R^2}e^{i\lambda t}(\lambda-\phi(\xi))^{k-1}\theta(\lambda-\phi(\xi))F(\xi,\lambda) d\xi d\lambda\right\|_{H_t^{\frac{s+1}{3}}(\R)}\\
\lesssim & \sum^\infty_{k=0}\frac{1}{k!}\left\|\langle\lambda\rangle^{\frac{s+1}{3}}\int_{|\lambda-\phi(\xi)|\leq1}F(\xi,\lambda)d\xi d\lambda\right\|_{L^2_\lambda}\\
\lesssim & \left[\int_\R\langle\lambda\rangle^{\frac{2s+2}{3}}\Big(\int_{|\lambda-\phi(\xi)|\leq1} \langle\xi\rangle^{-2s}d\xi\Big)\Big(\int_{|\lambda-\phi(\xi)|\leq1}\langle\xi\rangle^{2s}\left(F\right)^2d\xi\Big)d\lambda \right]^\frac{1}{2}\\
\lesssim& \sup_{\lambda} \Big(\langle\lambda\rangle^{\frac{2s+2}{3}}\int_{|\lambda-\phi(\xi)|\leq1} \langle\xi\rangle^{-2s}d\xi\Big)^{\frac{1}{2}}  \left\|\check{F}\right\|_{X^{s,-b}}
:= K_1 \left\|\check{F}\right\|_{X^{s,-b}}.
\end{align*}
Then, it  suffices to bound the term $K_1$. To handle that, we consider the domain of the integral, $|\lambda-\phi(\xi)|\leq1$,   in either $|\xi|, |\lambda| \lesssim 1$ or $|\xi|, |\lambda| \gg 1$. The former case is readily to check. The later case can be verified by substituting $\mu=\phi(\xi)\sim |\xi|^3$, it follows,
\[\langle\lambda\rangle^{\frac{2s+2}{3}}\int_{-1+|\lambda|<|\mu|<1+|\lambda|} \langle\mu\rangle^{-\frac{2s+2}{3}}d\mu\lesssim \langle\lambda\rangle^{\frac{2s+2}{3}}\langle\lambda\rangle^{-\frac{(2s+2)}{3}}\lesssim 1, \quad \mbox{for }  |\lambda| \gg 1.\]

For $A_2$, one has,
\begin{align}
\|\eta(t)A_2\|_{H_t^{\frac{s+1}{3}}(\R)}\lesssim & \left\|\int_{\R^2}e^{it\lambda} \frac{\theta^c(\lambda-\phi(\xi))}{\lambda-\phi(\xi)}F(\xi,\lambda)d\xi d\lambda\right\|_{H_t^{\frac{s+1}{3}}(\R)}\nonumber\\
\lesssim &\left\|\langle\lambda\rangle^{\frac{s+1}{3}} \int_\R\frac{\theta^c(\lambda-\phi(\xi))}{\lambda-\phi(\xi)}F(\xi,\lambda)d\xi\right\|_{L^2_\lambda}\nonumber\\
\lesssim & \left\|\langle\lambda\rangle^{\frac{s+1}{3}} \int_\R\frac{1}{\langle\lambda-\phi(\xi)\rangle}F(\xi,\lambda)d\xi\right\|_{L^2_\lambda}\label{A2}\\
\lesssim & \bigg[\int_\R \langle\lambda\rangle^{\frac{2s+2}{3}}\Big(\int_\R\frac{1}{\langle\xi\rangle^{2s}\langle\lambda-\phi(\xi)\rangle^{2-2b}}d\xi\Big)
\bigg(\int_\R\langle\xi\rangle^{2s}\nonumber\\
&\cdot\langle\lambda-\phi(\xi)\rangle^{-2b}\left(F(\xi,\lambda)\right)^2d\xi\bigg) d\lambda\bigg]^\frac{1}{2}\nonumber\\
\lesssim &\sup_{\lambda}\left(\langle\lambda\rangle^{\frac{2s+2}{3}}\int_\R\frac{1}{\langle\xi\rangle^{2s}\langle\lambda-\phi(\xi)\rangle^{2-2b}}
d\xi\right)^{\frac{1}{2}} \|\check{F}\|_{X^{s,-b}}\nonumber\\
:= &K_2  \|\check{F}\|_{X^{s,-b}}\nonumber.
\end{align}
Then, it suffices to bound the term $K_2$.   Let us set $z=\phi(\xi)$ and $\frac{d\xi}{d z}\sim |\xi|^{-2}\sim |z|^{-\frac23}$ as $\xi\rightarrow \infty$. According to Lemma \ref{tz1} with $s\leq 0$ and $b<\frac12$,
\begin{equation}\label{A21}
K_2^2\lesssim \langle\lambda\rangle^{\frac{2s+2}{3}}  \int_\R \frac{1}{\langle z\rangle^{\frac{2s+2}{3}}\langle\lambda\pm z\rangle^{2-2b}}dz\lesssim \langle\lambda\rangle^{\frac{2s+2}{3}}\langle\lambda\rangle^{-\frac{2s+2}{3}}\lesssim 1.
\end{equation}

For $A_3$, substituting $\mu=\phi(\xi)$, for $b<\frac12$, it follows
\begin{align*}
\|\eta(t)A_3\|_{H_t^{\frac{s+1}{3}}(\R)}\lesssim &\left\|\int_{\R^2}\frac{e^{it\mu}}{\lambda-\mu}\theta^c(\lambda-\mu)F(\mu,\lambda) \mu^{-\frac23}d\mu d\lambda \right\|_{H_t^{\frac{s+1}{3}}(\R)}\\
\lesssim& \left\|\langle\mu\rangle^{\frac{s+1}{3}} \int_{\R}\frac{\theta^c(\lambda-\mu)}{\lambda-\mu}F(\mu,\lambda) \mu^{-\frac23} d\lambda\right\|_{L^2_\mu}\\
\lesssim& \left(\int_\R \langle\mu\rangle^{\frac{2s+2}{3}}\mu^{-\frac43}  \int_{\R}\frac{1}{\langle\lambda-\mu\rangle^{2b}}F^2(\mu,\lambda) d\lambda \int_\R  \frac{1}{\langle\lambda-\mu\rangle^{2-2b}} d\lambda d\mu\right)^{\frac12}\\
\lesssim& \left(\int_{\R^2} \langle\xi\rangle^{2s} \frac{1}{\langle\lambda-\phi(\xi)\rangle^{2b}}F^2(\xi,\lambda) d\lambda d\xi\right)^{\frac12}\lesssim \left\|\check{F}\right\|_{X^{s,-b}}.
\end{align*}

Next, we turn to show the estimates holds for $j=1$. Similar to the previous proof, with the same definition of $F(\xi,\lambda)$, let us write
  \begin{align*}
  &\int^t_0 \partial_x [W_\R(0,g)](x,t-\tau)d\tau\Big|_{x=0}\\
  =&\int^t_0\int_\R \frac{\xi\left(e^{i(t-\tau)\phi(\xi)}-e^{-i(t-\tau)\phi(\xi)}\right)}{2i\phi(\xi)} \xi^2\widehat{g}^x(\xi,\tau)d\xi d\tau\\
  =&\int^t_0\int_\R  \frac{\xi\left(e^{i(t-\tau)\phi(\xi)}-e^{-i(t-\tau)\phi(\xi)}\right)}{2i\phi(\xi)}\xi^2\left(\int_\R e^{i\lambda \tau}\widehat{g}(\xi,\lambda)d\lambda\right) d\xi d\tau\\
  :=&\frac{1}{2i}(W_1+W_2),
   \end{align*}
where
  \begin{align*}
W_1=-i\int_{\R^2}\frac{\xi\left(e^{it\lambda}-e^{it\phi(\xi)}\right)}{\lambda-\phi(\xi)}F(\xi,\lambda)d\lambda d\xi
\end{align*}
and
\begin{equation*}
W_2=-i\int_{\R^2}\frac{\xi\left(e^{it\lambda}-e^{it\phi(\xi)}\right)}{\lambda-\phi(\xi)}F(\xi,\lambda)d\lambda d\xi.
\end{equation*}
Again, it suffices to verify the estimate for $W_1$ with
\begin{align*}
W_1=&-i\int_{\R^2}\frac{\xi\left(e^{it\lambda}-e^{it\phi(\xi)}\right)}{\lambda-\phi(\xi)}\theta(\lambda-\phi(\xi))F(\xi,\lambda)d\lambda d\xi\\
-i&\int_{\R^2}\frac{\xi e^{it\lambda}}{\lambda-\phi(\xi)}\theta^c(\lambda-\phi(\xi))F(\xi,\lambda)d\lambda d\xi\\
+i&\int_{\R^2}\frac{\xi e^{it\phi(\xi)}}{\lambda-\phi(\xi)}\theta^c(\lambda-\phi(\xi))F(\xi,\lambda)d\lambda d\xi\\
:&=-i(B_1+B_2-B_3).
\end{align*}
The proof for $B_1$ and $B_3$ are similar to the one for $A_1$ and $A_3$. For $B_2$, it follows from similar process as in \eqref{A2} that
\begin{align*}
\|\eta(t)B_2\|_{H_t^{\frac{s}{3}}(\R)}\lesssim  \left\|\langle\lambda\rangle^{\frac{s}{3}} \int_\R\frac{\xi}{\langle\lambda-\phi(\xi)\rangle}F(\xi,\lambda)d\xi\right\|_{L^2_\lambda}.
\end{align*}
For $\xi\in D$, since $|\xi|^3\gg |\lambda|$ and $|\xi|\gtrsim 1$, $\langle\lambda-\phi(\xi)\rangle\sim |\xi|^{3}$, then the result follows. For $\xi\not\in D$, that is, either $|\xi|^3\lesssim |\lambda|$ or $|\xi|\lesssim 1$.  Repeating the process in $A_2$ (c.f. \eqref{A2})  arrives at the bound,
\[
\|\eta(t)B_2\|_{H_t^{\frac{s+1}{3}}(\R)}\lesssim \sup_{\lambda}\left(\langle\lambda\rangle^{\frac{2s}{3}}\int_\R\frac{|\xi|^2}{\langle\xi\rangle^{2s}\langle\lambda-\phi(\xi)\rangle^{2-2b}}
d\xi\right)^{\frac{1}{2}} \|\check{F}\|_{X^{s,-b}}.\]
\begin{itemize}
  \item If  $|\xi|^3\lesssim |\lambda|$,   it leads to $|\xi|^2\lesssim |\lambda|^{\frac23}$ and
\[\langle\lambda\rangle^{\frac{2s}{3}}\int_\R\frac{|\xi|^2}{\langle\xi\rangle^{2s}\langle\lambda-\phi(\xi)\rangle^{2-2b}}
d\xi\lesssim \langle\lambda\rangle^{\frac{2s+2}{3}}\int_\R\frac{1}{\langle\xi\rangle^{2s}\langle\lambda-\phi(\xi)\rangle^{2-2b}}
d\xi.\]
The bound follows exact same as \eqref{A21}.
  \item If  $|\xi|\lesssim 1$,  it leads to  $|\xi|^2\lesssim 1$ and
\[\langle\lambda\rangle^{\frac{2s}{3}}\int_\R\frac{|\xi|^2}{\langle\xi\rangle^{2s}\langle\lambda-\phi(\xi)\rangle^{2-2b}}
d\xi\lesssim \langle\lambda\rangle^{\frac{2s}{3}}\int_\R\frac{1}{\langle\xi\rangle^{2s}\langle\lambda-\phi(\xi)\rangle^{2-2b}}
d\xi.\]
 The bound is obtained  similar to \eqref{A21}.
\end{itemize}

Similar proof can be carried out on  the estimate for $j=2$.

Finally, let us consider the proof for $j=3,4,5$. It has shown (c.f.  Lemma 2.12 in \cite{96}) that
\begin{equation*}
 \partial_x^3\left(\int^t_0[W_\R(0,g)](x,t-\tau)d\tau\right)=\partial_t\left(\int^t_0[W_\R(0,g)](x,t-\tau)d\tau\right),
\end{equation*}
thus,
\begin{equation*}
\partial_x^4\left(\int^t_0[W_\R(0,g)](x,t-\tau)d\tau\right)=\partial_t\left(\int^t_0[\partial_x W_\R(0,g)](x,t-\tau)d\tau\right),
\end{equation*}
\begin{equation*}
\partial_x^5\left(\int^t_0[W_\R(0,g)](x,t-\tau)d\tau\right)=\partial_t\left(\int^t_0[\partial_x^2 W_\R(0,g)](x,t-\tau)d\tau\right).
\end{equation*}
Since $\eta(t)$ is a cut-off function with $\eta\in C^\infty(\R)$, it follows that  for $j=3,4,5$,
\begin{align*}
  &\left\| \eta(t)\int^t_0 \partial_x^j W_{\mathbb{R}}(0,g)(x,t-\tau)d\tau\right\|_{H_t^{\frac{s-j+1}{3}}(\R)}\\
  = &\left\| \eta(t)\partial_t\left(\int^t_0 \partial_x^{j-3} W_{\mathbb{R}}(0,g)(x,t-\tau)d\tau\right)\right\|_{H_t^{\frac{s-j+1}{3}}(\R)}\\
  \lesssim & \left\| \eta(t)\int^t_0 \partial_x^{j-3} W_{\mathbb{R}}(0,g)(x,t-\tau)d\tau\right\|_{H_t^{\frac{s-j+4}{3}}(\R)},
\end{align*}
 these will reduce the proof to the case $j=0,1,2$. The proof is now complete.
\end{proof}
\section{Bilinear Estimates}
As pointed out in the introduction as well as in the remark of   Proposition \ref{k1}, estimates on $X^{s,b}$ for the forced problem only works when $b<\frac12$.  However, the bilinear estimate for the IVP of the sixth order Boussinesq equation on $X^{s,b}$  is  only valid for $\frac12<b<1$ (c.f. \cite{34}). Therefore, it is necessary to  derive  a proper bilinear estimate for $b<\frac12$ in order to establish a contraction mapping.
\begin{proposition}\label{t1}
For  $-\frac12< s\leq 0$, $\frac12-b>0 $ sufficient small,
\[\left\|\left(\frac{\xi^2\widehat{uv}}{\phi{(\xi})}\right)^\vee\right\|_{X^{s,-b}}\lesssim \|u\|_{X^{s,b}}\|v\|_{X^{s,b}}
.\]
\end{proposition}

\begin{proposition}\label{t2}
For $-\frac12<s\leq 0$, $\frac12 - b>0$ sufficient small,
\[\left\|\langle\tau\rangle^{\frac{s-j+1}{3}}\int_D\langle\xi\rangle^{j-3}\left|\frac{\xi^2\widehat{uv}}{\phi(\xi)}\right|d\xi\right\|_{L^2_\tau}
\lesssim \|u\|_{X^{s,b}} \|v\|_{X^{s,b}}, \quad j=1,2,\]
where $D:= \{\xi\in \R:|\xi|^3\gg |\tau|, |\xi|\gtrsim 1\}$.
\end{proposition}

\begin{proof}
\textbf{(Proposition \ref{t1})} For simplicity,  there is no harm to assume  $\beta=1$. In addition,
in order to have a better cancellation, the equivalence relation in norms is adapted for the norm in $X^{s,b}$ (c.f. Lemma \ref{equ}).
Let $u,v\in X^{s,b}$, $w\in X^{-s,b}$ be given, we  define
\[
f(\xi,\tau)=\langle\xi\rangle^{s}\langle|\tau|-|\xi|^3-\frac12|\xi|\rangle^b\widehat{u}(\xi,\tau),\ \ \ \ g(\xi,\tau)=\langle\xi\rangle^{s}\langle|^s\tau|-|\xi|^3-\frac12|\xi|\rangle^b\widehat{v}(\xi,\tau),
\]
\[
h(\xi,\tau)=\langle\xi\rangle^{-s}\langle|\tau|-|\xi|^3-\frac12|\xi|\rangle^b\overline{\widehat{w}(\xi,\tau)}.
\]
According to duality, it suffices to show
\begin{equation}\label{ww}
|W(f,g,h)|\lesssim \|f\|_{L^2_{\xi,\tau}}\|g\|_{L^2_{\xi,\tau}}\|h\|_{L^2_{\xi,\tau}},
\end{equation}
where
\begin{align}
W(f,g,h)
=&\int_{\R^4}\frac{|\xi|^2\langle\xi\rangle^{s}\langle\xi_1\rangle^{-s}\langle \xi-\xi_1\rangle^{-s}f(\xi_1,\tau_1)h(\xi,\tau)}{\phi(\xi)\langle|\tau|-|\xi|^3-\frac12|\xi|\rangle^b \langle|\tau_1|-|\xi_1|^3-\frac12|\xi_1|\rangle^b }\nonumber\\
&\cdot\frac{g(\xi-\xi_1,\tau-\tau_1)}{\langle|\tau-\tau_1|-|\xi-\xi_1|^3-\frac12|\xi-\xi_1|\rangle^b}d\xi d\tau d\xi_1 d\tau_1\nonumber\\
:=&\int_{\R^4}M (\xi,\xi_1,\tau,\tau_1) f(\xi_1,\tau_1)g(\xi-\xi_1,\tau-\tau_1)h(\xi,\tau)d\xi d\tau d\xi_1 d\tau_1\label{M3}.
\end{align}
To establish the estimate \eqref{ww},  we adapt   a similar argument in   Tzirakis \cite{97,95} (also, c.f. Farah \cite{34})   for the Boussiensq-type equations and analysis the six possible cases for the sign of $\tau$, $\tau_1$ and $\tau-\tau_1$,
\begin{description}
\item[I] $\tau_1\geq 0$, $\tau-\tau_1\geq 0 $;
\item[II] $\tau_1\geq 0$, $\tau-\tau_1\leq 0 $, $\tau\geq 0$;
\item[III] $\tau_1\geq 0$, $\tau-\tau_1\leq 0 $, $\tau\leq 0$;
\item[IV] $\tau_1\leq 0$, $\tau-\tau_1\leq 0 $;
\item[V] $\tau_1\leq 0$, $\tau-\tau_1\leq 0 $, $\tau\leq 0$;
\item[VI] $\tau_1\leq 0$, $\tau-\tau_1\leq 0 $, $\tau\geq 0$.
\end{description}

For the estimate in (I), according to the Cauchy-Schwartz inequality, it suffices to verify,
\begin{align}
\Big\|\int_{\R^2} M (\xi,\xi_1,\tau,\tau_1) f(\xi_1,\tau_1)g(\xi-\xi_1,\tau-\tau_1)d\xi_1 d\tau_1&\Big\|_{L^2_{\xi,\tau}}\label{W1}\\
&\lesssim \|f\|_{L^2_{\xi,\tau}}\|g\|_{L^2_{\xi,\tau}}\nonumber.
\end{align}
Moreover, Using  the Cauchy-Schwartz and Young's (c.f. Lemma \ref{young}) inequalities, the left-hand side of \eqref{W1} is bounded by
\begin{align*}
&\left\|\|M\|_{L^2_{\xi_1,\tau_1}}\| f(\xi_1,\tau_1)g(\xi-\xi_1,\tau-\tau_1)\|_{L^2_{\xi_1,\tau_1}}\right\|_{L^2_{\xi,\tau}}\\
\lesssim& \left(\sup_{\xi,\tau}\|M\|_{L^2_{\xi_1,\tau_1}}\right) \| f(\xi_1,\tau_1)g(\xi-\xi_1,\tau-\tau_1)\|_{L^2_{\xi,\tau,\xi_1,\tau_1}}\\
=&  \left(\sup_{\xi,\tau}\|M\|_{L^2_{\xi_1,\tau_1}}\right) \|f^2*g^2\|^{\frac12}_{L^1_{\xi,\tau}}\lesssim  \left(\sup_{\xi,\tau}\|M\|_{L^2_{\xi_1,\tau_1}}\right)\|f\|_{L^2_{\xi,\tau}}\|g\|_{L^2_{\xi,\tau}}.
\end{align*}
Therefore, it suffices to bound the term $\|M\|_{L^2_{\xi_1,\tau_1}}$.  Set $r=-s\geq 0$, one has
\begin{align*}
\|M\|^2_{L^2_{\xi_1,\tau_1}}=\int_{\R^2}&\frac{\xi^4 \langle\xi\rangle^{-2r} }{\phi^2(\xi)\langle\tau-|\xi|^3-\frac12|\xi|\rangle^{2b}}\\
&\cdot \frac{\langle\xi-\xi_1\rangle^{2r}\langle\xi_1\rangle^{2r} }{\langle\tau-\tau_1
-|\xi-\xi_1|^3-\frac12|\xi-\xi_1|\rangle^{2b}\langle\tau_1-|\xi_1|^3-\frac12|\xi_1|\rangle^{2b}}d\xi_1d\tau_1
\end{align*}
Since $\langle p\rangle \langle q\rangle \gtrsim \langle p+q\rangle$, $\xi^4/\phi^2(\xi)\lesssim \langle\xi\rangle^{-2}$, $\frac12-b>0$ sufficient small and Lemma \ref{tz1}, it arrives at
\begin{align*}
\|M\|^2_{L^2_{\xi_1,\tau_1}}\lesssim&  \int_{\R^2}\frac{ \langle\xi\rangle^{-2r-2}\langle\xi_1\rangle^{2r}\langle\xi-\xi_1\rangle^{2r} d\xi_1d\tau_1}{\langle\tau_1-|\xi_1|^3-\frac12|\xi_1|\rangle^{2b}\langle\tau_1+|\xi-\xi_1|^3-|\xi|^3+\frac12|\xi-\xi_1|-\frac12|\xi|\rangle^{2b}}\\
\lesssim &  \int_{\R^2}\frac{ \langle\xi\rangle^{-2r-2}\langle\xi_1\rangle^{2r}\langle\xi-\xi_1\rangle^{2r} }{\langle|\xi-\xi_1|^3+|\xi_1|^3-|\xi|^3+\frac12|\xi-\xi_1|+\frac12|\xi_1|-\frac12|\xi|\rangle^{2b}}d\xi_1\\
=&  \int_{A_1+A_2}\frac{ \langle\xi\rangle^{-2r-2}\langle\xi_1\rangle^{2r}\langle\xi-\xi_1\rangle^{2r} }{\langle|\xi-\xi_1|^3+|\xi_1|^3-|\xi|^3+\frac12|\xi-\xi_1|+\frac12|\xi_1|-\frac12|\xi|\rangle^{2b}}d\xi_1\\
:=&N_1+N_2,
\end{align*}
where
\begin{equation*}
A_1=\{\xi_1\in \R: \langle\xi_1\rangle\langle\xi-\xi_1\rangle\lesssim\langle\xi\rangle\}\ \ \ \mbox{and}\ \ \  A_2=\{\xi_1\in \R: \langle\xi_1\rangle\langle\xi-\xi_1\rangle\gg\langle\xi\rangle\}.
\end{equation*}

 In $A_1$,  one has
\begin{align*}
N_1=&\int_{A_1}
\frac{ \langle\xi\rangle^{-2r-2}\langle\xi_1\rangle^{2r}\langle\xi-\xi_1\rangle^{2r} }{\langle|\xi-\xi_1|^3+|\xi_1|^3-|\xi|^3+\frac12|\xi-\xi_1|+\frac12|\xi_1|-\frac12|\xi|\rangle^{2b}}d\xi_1 \\
\lesssim & \langle\xi\rangle^{-2}\int \frac{1}{\langle|(\xi-\xi_1)^3+\xi_1^3|-|\xi|^3+\frac12|(\xi-\xi_1)+\xi_1|-\frac12|\xi|\rangle^{2b}}d\xi_1\\
=& \langle\xi\rangle^{-2}\int \frac{1}{\langle|\xi|(\xi^2-3\xi\xi_1+3\xi_1^2)-|\xi|\xi^2\rangle^{2b}}d\xi_1\\
\lesssim& \langle\xi\rangle^{-2}\int \frac{1}{\langle3\xi\xi_1(\xi-\xi_1)\rangle^{1-}}d\xi_1,
\end{align*}
since $\xi^2-3\xi\xi_1+3\xi_1^2\geq 0$. Set $x=3\xi\xi_1(\xi-\xi_1)$, it follows,
\[\xi_1=\frac{3\xi^2\pm \sqrt{9\xi^4-12\xi x}}{6\xi},\ \ \ \ \  \mbox{and}\ \ \ \ \ dx=\pm2|\xi|^{\frac12}\sqrt{9\xi^3-12x}d\xi_1.\]
According to Lemma \ref{tz2}, one has the following bound,
\[N_1\lesssim \langle\xi\rangle^{-\frac52} \int \frac{1}{\langle x\rangle^{1-}\sqrt{9\xi^3-12x}}dx\lesssim \langle\xi\rangle^{-\frac52} \langle\xi^3\rangle^{-\frac12+}\lesssim \langle\xi\rangle^{-4+}.\]

In $A_2$,  let us  consider the integral
\[N_2= \int_{A_2}\frac{ \langle\xi\rangle^{-2r-2}\langle\xi_1\rangle^{2r}\langle\xi-\xi_1\rangle^{2r} }{\langle|\xi-\xi_1|^3+|\xi_1|^3-|\xi|^3+\frac12|\xi-\xi_1|+\frac12|\xi_1|-\frac12|\xi|\rangle^{2b}}d\xi_1,\]
 in following split regions:
\begin{itemize}
\item $|\xi_1|\gg |\xi|$,
\item $|\xi|\gg |\xi_1| \gg 1$,
\item  $|\xi|\sim|\xi_1|\gg 1$.
\end{itemize}
Furthermore, for each region above, the estimate is considered  in  different  cases below,
\begin{description}
  \item[$(a)$] $B_1:=\{(\xi,\xi_1):\xi-\xi_1\geq 0, \xi_1\geq0, \xi\geq 0\}$;
  \item[$(b)$] $B_2:=\{(\xi,\xi_1):\xi-\xi_1\geq 0, \xi_1\leq0, \xi\geq 0\}$;
  \item[$(c)$] $B_3:=\{(\xi,\xi_1):\xi-\xi_1\geq 0, \xi_1\leq0, \xi\leq 0\}$;
  \item[$(d)$] $B_4:=\{(\xi,\xi_1):\xi-\xi_1\leq 0, \xi_1\leq0, \xi\leq 0\}$;
  \item[$(e)$] $B_5:=\{(\xi,\xi_1):\xi-\xi_1\leq 0, \xi_1\geq0, \xi\leq 0\}$;
  \item[$(f)$] $B_6:=\{(\xi,\xi_1):\xi-\xi_1\leq 0, \xi_1\geq0, \xi\geq 0\}$.
\end{description}
Notice that cases for $\{\xi-\xi_1\geq 0, \xi_1\geq 0, \xi\leq0 \}$ and $\{\xi-\xi_1\leq 0, \xi_1\leq 0, \xi\geq0 \}$ do not exist for $\xi$ and $\xi_1$ nonzero. In addition,  one only need to consider the estimate in cases $B_1$, $B_2$ and $B_3$ since the rest cases are equivalence to those.\\
\textbf{Case 1.} For $|\xi_1|\gg |\xi|$, one  has $\langle\xi-\xi_1\rangle\sim \langle\xi_1\rangle$. Notice that $B_1$ cannot contribute to this region. For both $B_2$ and $B_3$, it yields,
\begin{align*}
 &\langle|\xi-\xi_1|^3+|\xi_1|^3-|\xi|^3+\frac12|\xi-\xi_1|+\frac12|\xi_1|-\frac12|\xi|\rangle^{2b}\\
 =&\langle(\xi-\xi_1)^3-\xi_1^3\pm\xi^3+\frac12(\xi-\xi_1)-\frac12\xi_1\pm\frac12\xi\rangle^{2b} \sim\langle\xi_1^3+\xi_1\rangle^{1-}.
\end{align*}
Therefore,
\[N_2\lesssim \int_{|\xi_1|\gg |\xi|}\frac{ \langle\xi\rangle^{-2r-2}\langle\xi_1\rangle^{4r} }{\langle\xi_1^3+\xi_1\rangle^{1-}}d\xi_1. \]
 If $|\xi_1|\leq 1$, the bound is readily to check. If $|\xi_1|> 1$, then,
 \[N_2\lesssim \langle\xi \rangle^{-2r-2}\int_{|\xi_1|\gg|\xi| } \langle\xi_1\rangle^{4r-3+}d\xi_1\lesssim \langle\xi\rangle ^{2r-4+},\]
 which is  bounded if $r<  \frac12$.

 \textbf{Case 2.} For $|\xi|\gg |\xi_1| \gg 1$, one has $\langle\xi-\xi_1\rangle\sim \langle\xi\rangle$. For both $B_1$ and $B_2$, it yields,
\begin{align*}
 &\langle|\xi-\xi_1|^3+|\xi_1|^3-|\xi|^3+\frac12|\xi-\xi_1|+\frac12|\xi_1|-\frac12|\xi|\rangle^{2b}\\
 =&\langle(\xi-\xi_1)^3\pm\xi_1^3-\xi^3+\frac12(\xi-\xi_1)\pm\frac12\xi_1-\frac12\xi\rangle^{2b}\sim\langle\xi^2\xi_1\rangle^{1-} \sim \langle\xi\rangle^{2-}\langle\xi_1\rangle^{1-}.
\end{align*}
 Thus,
 \[N_2\lesssim \langle\xi \rangle^{-4+}\int_{|\xi_1|\ll |\xi|} \langle\xi_1\rangle^{2r-1+}d\xi_1\lesssim \langle\xi\rangle^{2r-4+},\]
 which is bounded if $r< 2$. While for $B_3$, it yields,
\begin{align*}
 &\langle|\xi-\xi_1|^3+|\xi_1|^3-|\xi|^3+\frac12|\xi-\xi_1|+\frac12|\xi_1|-\frac12|\xi|\rangle^{2b}\\
 =&\langle(\xi-\xi_1)^3-\xi_1^3+\xi^3+\frac12(\xi-\xi_1)-\frac12\xi_1+\frac12\xi\rangle^{2b}\sim\langle\xi^3\rangle^{1-} \sim \langle\xi\rangle^{3-}.
\end{align*}

 Thus, one has the  bound if $r< 2$, since
 \[N_2\lesssim \langle\xi\rangle^{-5+}\int_{|\xi_1|\ll |\xi|} \langle\xi_1\rangle^{2r}d\xi_1\lesssim |\xi|^{2r-4+}.\]
 \textbf{Case 3.} For $|\xi|\sim|\xi_1|\gg 1$,  recall that, in $A_2$, $ \langle\xi_1\rangle\langle\xi-\xi_1\rangle\gg\langle\xi\rangle$, this gives $|\xi-\xi_1|\gg1$. For all of $B_1$, $B_2$ and $B_3$, it yields,
 \begin{align*}
 &\langle|\xi-\xi_1|^3+|\xi_1|^3-|\xi|^3+\frac12|\xi-\xi_1|+\frac12|\xi_1|-\frac12|\xi|\rangle^{2b}\\
 \gtrsim &\langle|(\xi-\xi_1)^3+\xi_1^3|-|\xi|^3\rangle^{2b} \sim \langle\xi\xi_1(\xi-\xi_1)\rangle^{2b}\sim \langle\xi_1\rangle^{2-}\langle\xi-\xi_1\rangle^{1-}.
 \end{align*}
  Thus, for $r<\frac12$, i.e. $2r-1<0$, the bound follows as below,
 \[N_2\lesssim \langle\xi\rangle^{-2r-2}\int_{|\xi_1|\sim |\xi|} \langle\xi_1\rangle^{2r-2+}\langle\xi-\xi_1\rangle^{2r-1+} d\xi_1\lesssim  \langle\xi\rangle^{-3+}.\]
Therefore,  the desired estimate in (I) has shown.

 Next,  similar idea is applied to process (\textrm{II}). By exchange the role of $(\xi,\tau)$ and $(\xi_1,\tau_1)$,  it suffices to verify,
\[\left\|\int_{\R^2} M (\xi,\xi_1,\tau,\tau_1) g(\xi-\xi_1,\tau-\tau_1)h(\xi,\tau)d\xi d\tau\right\|_{L^2_{\xi_1,\tau_1}}\lesssim \|g\|_{L^2_{\xi_1,\tau_1}}\|h\|_{L^2_{\xi_1,\tau_1}},
\]
with $M$ defined in \eqref{M3}. Similar argument for (I) can be applied and it suffices to bound the following,
\begin{align*}
\|M\|^2_{L^2_{\xi,\tau}}\lesssim& \int_{A_3+A_4}\frac{ \langle\xi\rangle^{-2r-2}\langle\xi_1\rangle^{2r}\langle\xi-\xi_1\rangle^{2r} }{\langle|\xi-\xi_1|^3+|\xi_1|^3-|\xi|^3+\frac12|\xi-\xi_1|+\frac12|\xi_1|-\frac12|\xi|\rangle^{2b}}d\xi\\
:=& L_1+L_2,
\end{align*}
where
\begin{equation*}
A_3=\{\xi\in \R: \langle\xi_1\rangle\langle\xi-\xi_1\rangle\lesssim\langle\xi\rangle\}\ \ \ \mbox{and}\ \ \  A_4=\{\xi\in \R: \langle\xi_1\rangle\langle\xi-\xi_1\rangle\gg\langle\xi\rangle\}.
\end{equation*}

In $A_3$,  one has
\begin{align*}
L_1=&\int_{A_3}  \frac{ \langle\xi\rangle^{-2r-2}\langle\xi_1\rangle^{2r}\langle\xi-\xi_1\rangle^{2r}d\xi }{\langle|\xi-\xi_1|^3+|\xi_1|^3-|\xi|^3+\frac12|\xi-\xi_1|+\frac12|\xi_1|-\frac12|\xi|\rangle^{2b}} \\
\lesssim &\int_{A_3}  \frac{\langle\xi\rangle^{-2}d\xi}{\langle3\xi\xi_1(\xi-\xi_1)\rangle^{1-}}\lesssim \int_{A_3}  \frac{d\xi}{\langle3\xi\xi_1(\xi-\xi_1)\rangle^{1-}},
\end{align*}
 the rest of the proof is similar to the one for $A_1$ in (\textrm{I}).

In $A_4$,  the integration
\[L_2=\int_{A_4}  \frac{ \langle\xi\rangle^{-2r-2}\langle\xi_1\rangle^{2r}\langle\xi-\xi_1\rangle^{2r}d\xi }{\langle|\xi-\xi_1|^3+|\xi_1|^3-|\xi|^3+\frac12|\xi-\xi_1|+\frac12|\xi_1|-\frac12|\xi|\rangle^{2b}}, \]
 are considered by splitting the space $A_4$ into the same regions as for $A_2$  in (\textrm{I}):\\
 \textbf{Case 1.} For $|\xi_1|\gg |\xi|$, one has $\langle\xi-\xi_1\rangle\sim \langle\xi_1\rangle$. Notice that $B_1$ cannot contribute to this region. For both $B_2$ and $B_3$, it yields,
\begin{align*}
 &\langle|\xi-\xi_1|^3+|\xi_1|^3-|\xi|^3+\frac12|\xi-\xi_1|+\frac12|\xi_1|-\frac12|\xi|\rangle^{2b}\\
 =&\langle(\xi-\xi_1)^3-\xi_1^3\pm\xi^3+\frac12(\xi-\xi_1)-\frac12\xi_1\pm\frac12\xi\rangle^{2b} \sim\max\{\langle\xi_1^3\rangle^{1-}, \langle\xi_1\rangle^{1-}\}.
\end{align*}
For $|\xi_1|\leq 1$ the bound can be obtained directly. For $|\xi_1| >1$,
  we have the following bound with $r<2$,
 \[L_2\lesssim \langle\xi_1\rangle^{4r-3+}\int_{|\xi|\ll|\xi_1| }\langle\xi\rangle ^{-2r-2} d\xi\lesssim \langle\xi_1\rangle ^{2r-4+}.\]
\textbf{Case 2.} For $|\xi|\gg |\xi_1| \gg 1$, one has $\langle\xi-\xi_1\rangle\sim \langle\xi\rangle$. For both $B_1$ and $B_2$, it yields,
\begin{align*}
 &\langle|\xi-\xi_1|^3+|\xi_1|^3-|\xi|^3+\frac12|\xi-\xi_1|+\frac12|\xi_1|-\frac12|\xi|\rangle^{2b}\\
 =&\langle(\xi-\xi_1)^3\pm\xi_1^3-\xi^3+\frac12(\xi-\xi_1)\pm\frac12\xi_1-\frac12\xi\rangle^{2b}\sim\langle\xi^2\xi_1\rangle^{1-} \sim \langle\xi\rangle^{2-}\langle\xi_1\rangle^{1-}.
\end{align*}
 Thus,
 \[L_2\lesssim \langle\xi_1\rangle^{2r-1+}\int_{|\xi|\gg |\xi_1|} \langle\xi\rangle^{-4+}d\xi\lesssim \langle\xi_1\rangle^{2r-4+},\]
 which is bounded if $r<2$. While for $B_3$, it follows,
\begin{align*}
&\langle|\xi-\xi_1|^3+|\xi_1|^3-|\xi|^3+\frac12|\xi-\xi_1|+\frac12|\xi_1|-\frac12|\xi|\rangle^{2b}\\
\gtrsim&\langle(\xi-\xi_1)^3-\xi_1^3+\xi^3\rangle^{2b}\sim\langle\xi^3\rangle^{1-} \sim \langle\xi\rangle^{3-}.
\end{align*}
 Thus,
 \[L_2\lesssim \langle\xi_1\rangle^{2r}\int_{|\xi|\gg |\xi_1|}\langle\xi\rangle^{-5+} d\xi\lesssim \langle\xi_1\rangle^{2r-4+},\]
 which is bounded if $r<2$.\\
 \textbf{Case 3.} For $|\xi|\sim|\xi_1|\gg 1$, since in $A_4$, $\langle\xi_1\rangle\langle\xi-\xi_1\rangle\gg\langle\xi\rangle$, this gives $|\xi-\xi_1|\gg1$. For all of $B_1$, $B_2$ and $B_3$, it yields,
\begin{align*}
 &\langle|\xi-\xi_1|^3+|\xi_1|^3-|\xi|^3+\frac12|\xi-\xi_1|+\frac12|\xi_1|-\frac12|\xi|\rangle^{2b}\\
 \gtrsim &\langle|(\xi-\xi_1)^3+\xi_1^3|-|\xi|^3\rangle^{2b} \sim \langle\xi\xi_1(\xi-\xi_1)\rangle^{2b}\sim \langle\xi\rangle^{2-}\langle\xi-\xi_1\rangle^{1-}.
\end{align*}
 Thus, for $r<\frac12$, we have $2r-1<0$, and the bound follows,
 \[L_2\lesssim \int_{|\xi|\sim |\xi_1|} \langle\xi\rangle^{-4+}\langle\xi-\xi_1\rangle^{2r-1+} d\xi\lesssim \langle\xi_1\rangle^{-3+}.\]

 Finally,  the estimate in $(\textrm{III})$ can be established  by performing the change of variables
$(\xi_1,\tau_1)\rightarrow (\xi-\xi_1,\tau-\tau_1)$,   then it is reduced to the estimate in $(\textrm{I})$ for $M$.  The proof is now complete.
 \end{proof}

\begin{proof}
\textbf{(Proposition \ref{t2})}
Similar to Proposition \ref{t1}, by introducing
\[f(\xi,\tau)=\langle\xi\rangle^s\langle|\tau|-|\xi|^3-\frac12|\xi|\rangle^b\widehat{u}(\xi,\tau),\ \ \ g(\xi,\tau)=\langle\xi\rangle^s\langle|\tau|-|\xi|^3-\frac12|\xi|\rangle^b\widehat{v}(\xi,\tau),\]
and setting $r=-s\geq 0$, the desired estimate becomes
\begin{align}
\bigg\|\langle\tau\rangle^{\frac{-r-j+1}{3}}\int_{\Omega} & \frac{\langle\xi\rangle^{{j-4}}\langle\xi_1\rangle^{r}\langle\xi-\xi_1\rangle^rf(\xi_1,\tau_1)
g(\xi-\xi_1,\tau-\tau_1)d\xi_1 d\tau_1 d\xi}{\langle|\tau_1|-|\xi_1|^3-\frac12|\xi_1|\rangle^b\langle|\tau-\tau_1|-|\xi-\xi_1|^3-\frac12|\xi-\xi_1|\rangle^b}\bigg\|_{L^2_\tau}\nonumber\\
&\lesssim \|f\|_{L^2_{\xi,\tau}}\|g\|_{L^2_{\xi,\tau}},\label{M}
\end{align}
where $\Omega=\R^2\times D$.
Applying the Cauchy-Schwartz inequality to integral in $(\xi,\xi_1,\tau_1)$ space and the Young's inequality (c.f. Lemma \ref{young}) arrives at a bound,
\begin{align*}
&\left\|\|M\|_{L^2(\Omega)}\|f(\xi_1,\tau_1)g(\xi-\xi_1,\tau-\tau_1)\|_{L^2_{\xi,\xi_1,\tau_1}}\right\|_{L^2_\tau}\\
\lesssim& \sup_{\tau} \|M\|_{L^2(\Omega)} \|f^2*g^2\|^\frac12_{L^1_{\xi,\tau}}\\
\lesssim& \sup_{\tau} \|M\|_{L^2(\Omega)}\|f\|_{L^2_{\xi,\tau}}\|g\|_{L^2_{\xi,\tau}},
\end{align*}
where
\begin{equation}\label{M0}
M:=\frac{\langle\tau\rangle^{\frac{-r-j+1}{3}}\langle\xi\rangle^{{j-4}}\langle\xi_1\rangle^{r}
\langle\xi-\xi_1\rangle^r}{\langle|\tau_1|-|\xi_1|^3-\frac12|\xi_1|\rangle^b\langle|\tau-\tau_1|-|\xi-\xi_1|^3-\frac12|\xi-\xi_1|\rangle^b}.
\end{equation}

Following is the proof for $j=2$, and the one for $j=1$ can be obtained in a similar way, therefore ommited. In order to bound,
\begin{equation}\label{M1}
 \|M\|_{L^2(\Omega)}^2=\int_{\Omega} \frac{\langle\tau\rangle^{\frac{-2r-2}{3}}\langle\xi\rangle^{{-4}}\langle\xi_1\rangle^{2r}\langle\xi-\xi_1\rangle^{2r}d\xi d\xi_1 d\tau_1}
 {\langle|\tau_1|-|\xi_1|^3-\frac12|\xi_1|\rangle^{2b}\langle|\tau-\tau_1|-|\xi-\xi_1|^3-\frac12|\xi-\xi_1|\rangle^{2b}},
\end{equation}
 the integral domain, $\Omega$, is split into following regions:\\
\textbf{Case 1.} For $\tau_1(\tau-\tau_1)\geq 0$, since $r\geq 0$, one can drop the $\langle\tau\rangle $ term and apply Lemma \ref{tz1}. This leads to,
\[\|M\|_{L^2(\Omega)}^2\lesssim \int_{\R\times D} \frac{\langle\xi\rangle^{{-4}}\langle\xi_1\rangle^{2r}\langle\xi-\xi_1\rangle^{2r}}{\langle\tau\pm (|\xi_1|^3+|\xi-\xi_1|^3+\frac12|\xi_1|+\frac12|\xi-\xi_1|)\rangle^{4b-1}}d\xi d\xi_1 .\]
In addition,  for either $|\xi|\lesssim |\xi_1|$ or $|\xi|\gg |\xi_1|$, one always has $|\xi_1|^3+|\xi-\xi_1|^3\gtrsim \max{\{|\xi|^3,|\xi_1|^3\}}$ and $|\xi_1|+|\xi-\xi_1|\gtrsim \max{\{|\xi|,|\xi_1|\}}$. Recall that in the domain of $D$,  $|\xi|\gg \tau$ and $|\xi|\gtrsim 1$, then the above  bound is reduced to,
\[\|M\|_{L^2(\Omega)}^2\lesssim \int_{\R\times D} \frac{\langle\xi\rangle^{{-4}}\langle\xi_1\rangle^{2r}\langle\xi-\xi_1\rangle^{2r}}{(\max{\{\langle\xi\rangle,\langle\xi_1\rangle\}})^{12b-3}}d\xi d\xi_1.\]
One can separate the domain of the integral  into two regions:
 \[D_1=\{(\xi,\xi_1)\in D\times \R : |\xi_1|\lesssim|\xi|  \} ,\quad D_2=\{(\xi,\xi_1)\in D\times \R: |\xi_1|\gg|\xi| \}. \]
 In $D_1$, one has $\langle\xi_1\rangle \lesssim \langle\xi\rangle$ and  $\langle\xi-\xi_1\rangle \lesssim \langle\xi\rangle$, then
\begin{align*}
\int_{D_1}\frac{\langle\xi\rangle^{{-4}}\langle\xi_1\rangle^{2r}\langle\xi-\xi_1\rangle^{2r}}{(\max{\{\langle\xi\rangle,\langle\xi_1\rangle\}})^{12b-3}}d\xi d\xi_1\lesssim &\int_\R \langle\xi\rangle^{4r-4}\int_{|\xi_1|\lesssim |\xi|} \langle\xi_1\rangle^{-12b+3}d\xi_1d\xi\\
\lesssim &\langle\xi\rangle^{4r-12b+1},
\end{align*}
which is bounded  if $r<\frac54$ and $\frac12-b>0$ sufficient small.
    In $D_2$, one has $\langle\xi-\xi_1\rangle \sim \langle\xi_1\rangle$, then
\begin{align*}
\int_{D_2}\frac{\langle\xi\rangle^{{-4}}\langle\xi_1\rangle^{2r}\langle\xi-\xi_1\rangle^{2r}}{(\max{\{\langle\xi\rangle,\langle\xi_1\rangle\}})^{12b-3}}d\xi d\xi_1\lesssim& \int_\R \langle\xi\rangle^{-4}\int_{|\xi_1|\gg |\xi|} \langle\xi_1\rangle^{4r-12b+3}d\xi_1d\xi\\
\lesssim& \langle\xi\rangle^{4r-12b+1},
\end{align*}
which is bounded if $r<\frac12$ and $\frac12-b>0$ sufficient small.\\
\textbf{Case 2.} For $\tau_1(\tau-\tau_1)<0$ and $|\xi_1|\lesssim | \xi|$,  similar to  the Case 1, by dropping the $\langle\tau\rangle$ term,   it leads to
\[\|M\|_{L^2(\Omega)}^2\lesssim\int_{\R\times D} \frac{\langle\xi\rangle^{{-4}}\langle\xi_1\rangle^{2r}\langle\xi-\xi_1\rangle^{2r}}{\langle\tau\pm (|\xi_1|^3-|\xi-\xi_1|^3+\frac12|\xi_1|-\frac12|\xi-\xi_1|)\rangle^{4b-1}}d\xi d\xi_1 .\]
The integral is estimated in the following two regions:
 \[D_3=\{(\xi,\xi_1)\in D\times \R : \xi_1(\xi-\xi_1)\geq 0  \} ,\quad D_4=\{(\xi,\xi_1)\in D\times \R: \xi_1(\xi-\xi_1)<0\}. \]

 In $D_3$, by changing the variable $z=2\xi_1^3-3\xi_1^2\xi+3\xi_1\xi^2-\xi^3+\xi_1-\frac12 \xi$ with $|z|\lesssim \max{\{|\xi|^3,|\xi|\}}$ and
 \[dz = (6\xi_1^2-6\xi_1\xi+3\xi^2+1)d \xi_1=\Big(3(\xi_1-\xi)^2+3\xi_1^2+1\Big)d\xi_1,\]
one has,
\begin{align*}
&\int_{D_3} \frac{\langle\xi\rangle^{{-4}}\langle\xi_1\rangle^{2r}\langle\xi-\xi_1\rangle^{2r}}{\langle\tau\pm (|\xi_1|^3-|\xi-\xi_1|^3+\frac12|\xi_1|-\frac12|\xi-\xi_1|)\rangle^{4b-1}}d\xi_1 d\xi\\
\lesssim & \int_{D_3} \frac{\langle\xi\rangle^{{-4}}\langle\xi_1\rangle^{2r}\langle\xi-\xi_1\rangle^{2r}}{\langle\tau\pm (2\xi_1^3-3\xi_1^2\xi+3\xi_1\xi^2-\xi^3+\xi_1-\frac12 \xi)\rangle^{4b-1}}d\xi_1 d\xi\\
\lesssim & \int_{D\times \{|z|\lesssim \max{\{|\xi|^3,|\xi|\}}\}} \frac{\langle\xi\rangle^{{-4}}\langle\xi_1\rangle^{2r}\langle\xi-\xi_1\rangle^{2r}}{(3\xi_1^2+3(\xi-\xi_1)^2+1)\langle\tau\pm z\rangle^{4b-1}}dz d\xi\\
\lesssim & \int_{D\times \{|z|\lesssim \max{\{|\xi|^3,|\xi|\}}\}} \frac{\langle\xi\rangle^{{-4}}\langle\xi_1\rangle^{2r}\langle\xi-\xi_1\rangle^{2r}}{(\xi_1^2+(\xi-\xi_1)^2)\langle\tau\pm z\rangle^{4b-1}}dz d\xi\\
\lesssim &  \int_{D\times \{|z|\lesssim \max{\{|\xi|^3,|\xi|\}}\}} \frac{1}{\langle\xi\rangle^{{4}} \langle\tau\pm z\rangle^{4b-1} }dzd\xi,
\end{align*}
since for $r<\frac12$, $|\xi|\gtrsim 1$ and $|\xi_1|\lesssim |\xi|$,
\begin{align*}
\frac{\langle\xi_1\rangle^{2r}\langle\xi-\xi_1\rangle^{2r}}{\xi_1^2+(\xi-\xi_1)^2}\lesssim 1.
\end{align*}
Recall that $|\tau|\ll |\xi|^3$ in $D$, then,
\begin{align*}
&\int_{D\times \{|z|\lesssim \max{\{|\xi|^3,|\xi|\}}\}} \frac{1}{\langle\xi\rangle^{{4}} \langle\tau\pm z\rangle^{4b-1} }dzd\xi\\
\lesssim&\int_{D}\langle\xi\rangle^{-4}\int_{|z|\lesssim\max{\{|\xi|^3,|\xi|\}}} \frac{1}{ \langle\tau\pm z\rangle^{4b-1}} dzd\xi\\
\lesssim &\int \langle\xi\rangle^{-4}\langle\max{\{|\xi|^3,|\xi|\}}\rangle^{2-4b}d\xi,
\end{align*}
which is bounded for $\frac12-b>0$ sufficient small.

  In $D_4$, by changing variable $z=3\xi\xi_1^2-3\xi^2\xi_1+\xi^3+\frac12\xi$ with $|z|\lesssim \max{\{|\xi|^3,|\xi|\}}$ and $dz=3\xi(2\xi_1-\xi)d\xi_1$, one has,
\begin{align*}
&\int_{D_4} \frac{\langle\xi\rangle^{{-4}}\langle\xi_1\rangle^{2r}\langle\xi-\xi_1\rangle^{2r}}{\langle\tau\pm (|\xi_1|^3-|\xi-\xi_1|^3+\frac12|\xi_1|-\frac12|\xi-\xi_1|)\rangle^{4b-1}}d\xi_1 d\xi\\
\lesssim & \int_{D_4} \frac{\langle\xi\rangle^{{-4}}\langle\xi_1\rangle^{2r}\langle\xi-\xi_1\rangle^{2r}}{\langle\tau\pm (3\xi\xi_1^2-3\xi^2\xi_1+\xi^3+\frac12\xi)\rangle^{4b-1}}d\xi_1 d\xi\\
\lesssim & \int_{D\times \{|z|\lesssim \max{\{|\xi|^3,|\xi|\}}\}} \frac{\langle\xi\rangle^{{-4}}\langle\xi_1\rangle^{2r}\langle\xi-\xi_1\rangle^{2r}}{|\xi||2\xi_1-\xi|\langle\tau\pm z\rangle^{4b-1}}dz d\xi.
\end{align*}
Notice that $\xi_1(\xi-\xi_1)<0$ gives $|2\xi_1-\xi|>|\xi|$. In addition,  if $|\xi_1|\lesssim |\xi|$ then $|\xi-\xi_1| \lesssim |\xi|$, thus $\langle\xi-\xi_1\rangle\lesssim \langle\xi\rangle$.  Therefore, one has the bound
\begin{align*}
&\int_{D\times \{|z|\lesssim |\xi|^3\}} \frac{\langle\xi\rangle^{{-4}}\langle\xi_1\rangle^{2r}\langle\xi-\xi_1\rangle^{2r}}{|\xi(2\xi_1-\xi)|\langle\tau\pm z\rangle^{4b-1}}dz d\xi\\
\lesssim &\int_{D\times \{|z|\lesssim \max{\{|\xi|^3,|\xi|\}}\}} \frac{\langle\xi\rangle^{{-4}}\langle\xi\rangle^{2r}\langle\xi\rangle^{2r}}{|\xi|^2\langle\tau\pm z\rangle^{4b-1}}dz d\xi \\
\lesssim & \int_{D} \langle\xi\rangle^{4r-6}\int_{|z|\lesssim \max{\{|\xi|^3,|\xi|\}}}  \frac{1}{ \langle\tau\pm z\rangle^{4b-1}} dzd\xi\\
\lesssim& \int_\R \langle\xi\rangle^{4r-6}\langle\max{\{|\xi|^3,|\xi|\}}\rangle^{2-4b}d\xi,
\end{align*}
which is again bounded if $r<\frac12$ and $b-\frac12<0$ sufficient small.\\
\textbf{Case 3.} For  $\tau_1(\tau-\tau_1)<0$ and $|\xi_1|\gg |\xi|$,  the estimate for (\ref{M}) will be established directly. According to duality, it suffices to show that
\[\int_{\Omega_1 } M f(\xi_1,\tau_1)g(\xi-\xi_1,\tau-\tau_1)h(\tau)d\xi d\tau d\xi_1 d\tau_1\lesssim \|f\|_{L^2_{\xi,\tau}}\|g\|_{L^2_{\xi,\tau}}\|h\|_{L^2_{\tau}},\]
where $M$ is defined in \eqref{M0} for $j=2 $ and $$ \Omega_1= D \times \{\xi_1\in \R:|\xi_1|\gg |\xi|\}\times \{ (\tau,\tau_1)\in \R^2:\tau_1(\tau-\tau_1)<0\}.$$ Applying the Cauchy-Schwartz, Holder and Young's inequalities, it arrives at
\begin{align*}
&\int_{\Omega_1} M f(\xi_1,\tau_1)g(\xi-\xi_1,\tau-\tau_1)h(\tau)d\xi d\tau d\xi_1 d\tau_1\\
\lesssim &\left\|\int_{\Omega_2}M g(\xi-\xi_1,\tau-\tau_1)h(\tau)d\xi d\tau\right\|_{L^2_{\xi_1,\tau_1}}\|f\|_{L^2_{\xi_1,\tau_1}}\\
\lesssim& \left\|\|M \langle\xi\rangle^{\frac12+}\|_{L^2(\Omega_2)} \|g(\xi-\xi_1,\tau-\tau_1)h(\tau)\langle\xi\rangle^{-\frac12-}\|_{L^2_{\xi,\tau}}\right\|_{L^2_{\xi_1,\tau_1}}\|f\|_{L^2_{\xi_1,\tau_1}}\\
\lesssim & \sup_{\xi_1,\tau_1} \|M \langle\xi\rangle^{\frac12+}\|_{L^2(\Omega_2)} \|g^2* (h(\tau)\langle\xi\rangle^{-\frac12-})^2\|^{\frac12}_{L^1_{\xi_1,\tau_1}}\\
\lesssim & \sup_{\xi_1,\tau_1} \|M \langle\xi\rangle^{\frac12+}\|_{L^2(\Omega_2)} \|\langle\xi\rangle^{-\frac12-}\|_{L^2_{\xi}}\|h\|_{L^2_{\tau}}\|f\|_{L^2_{\xi,\tau}}\|g\|_{L^2_{\xi,\tau}},
\end{align*}
where $\Omega_2=\{(\xi,\tau):{|\xi|^3\gg |\tau|, |\xi|\gtrsim 1,|\xi_1|\gg |\xi|}, \tau_1(\tau-\tau_1)<0\}$.
It remains to show that for $\tau_1>0$ and $\tau-\tau_1<0$,
\[ \|M \langle\xi\rangle^{\frac12+}\|_{L^2(\Omega_2)}= \int_{ \Omega_2}\frac{\langle\tau\rangle^{\frac{-2r-2}{3}}\langle\xi\rangle^{{-3+}}\langle\xi_1\rangle^{2r}\langle\xi-\xi_1\rangle^{2r}d\xi d\tau}
{\langle\tau_1-|\xi_1|^3-\frac12|\xi_1|\rangle^{2b}\langle-\tau+\tau_1-|\xi-\xi_1|^3-\frac12|\xi-\xi_1|\rangle^{2b}} \]
is bounded, since $\tau_1<0$ and $\tau-\tau_1>0$ can be established similarly. Notice that in $\Omega_2$, one has $\langle\xi-\xi_1\rangle\sim \langle\xi_1\rangle$, $|\xi-\xi_1|^3-|\xi_1|^3+\frac12|\xi-\xi_1|-\frac12|\xi_1|=\pm |\xi|(\xi^2-3\xi\xi_1+3\xi_1^2+1)$ and $\langle\tau\pm |\xi|(\xi^2-3\xi\xi_1+3\xi_1^2+1)\rangle\sim \langle\xi\xi_1^2\rangle\sim \langle\xi\rangle\langle\xi_1\rangle^2$. These lead to
\begin{align*}
&\int_{\Omega_2}\frac{\langle\tau\rangle^{\frac{-2r-2}{3}}\langle\xi\rangle^{{-3+}}\langle\xi_1\rangle^{2r}\langle\xi-\xi_1\rangle^{2r}}
{\langle\tau_1-|\xi_1|^3-\frac12|\xi_1|\rangle^{2b}\langle-\tau+\tau_1-|\xi-\xi_1|^3-\frac12|\xi-\xi_1|\rangle^{2b}}d\xi d\tau\\
\lesssim & \int_{\Omega_2}\frac{\langle\tau\rangle^{\frac{-2r-2}{3}}\langle\xi\rangle^{{-3+}}\langle\xi_1\rangle^{2r}\langle\xi_1\rangle^{2r}}
{\langle\tau+|\xi-\xi_1|^3-|\xi_1|^3+\frac12|\xi-\xi_1|-\frac12|\xi_1|\rangle^{2b}}d\xi d\tau\\
\lesssim & \langle\xi_1\rangle^{4r}\int_{|\xi|\ll |\xi_1|} \frac{\langle\xi\rangle^{{-3+}}}{\langle\xi\rangle^{2b}\langle\xi_1\rangle^{4b}}\int_{|\tau|\ll|\xi|^3}\langle\tau\rangle^{\frac{-2r-2}{3}}d\tau d\xi\\
\lesssim &\langle\xi_1\rangle^{4r-4b}\int_{|\xi|\ll |\xi_1|} \langle\xi\rangle^{-2b-3+}\langle\xi\rangle^{-2r+1} d\xi\lesssim \langle\xi_1\rangle^{2r-6b-1+},
\end{align*}
which is bounded if $r<\frac{1}{2}$ and $b<\frac12$. The proof is now complete.
\end{proof}

\section{Local Well-Posedness}

Now, we consider the  nonlinear problem:
\begin{equation}\label{hsb}
\begin{cases}
u_{tt}-u_{xx}+\beta u_{xxxx}-u_{xxxxxx}+(u^2)_{xx}=0,\hspace{0.3in}\mbox{for }x>0\mbox{, }t>0,\\
u(x,0)=\varphi (x), u_t(x,0)= \psi''(x),\\
u(0,t)=h_1(t), u_ {xx}(0,t)=h_2(t), u_{xxxx}(0,t)=h_3(t).
\end{cases}
\end{equation}
Let $Q_{s}$ be defined as below
\begin{equation*}
    Q_{s}=H^s(\mathbb{R}^+)\times H^{s-1}(\mathbb{R}^+)\times {\mathcal H}^s(\R^+),
\end{equation*}
where
\[{\mathcal H}^s(\R^+):= H^{\frac{s+2}{3}}(\R^+)\times H^{\frac{s}{3}}(\R^+)\times H^{\frac{s-2}{3}}(\R^+),\]
with its usual product topology.
Let $Y_{s,T}$ be the collection of
\begin{equation*}
    v\in C(0,T; H^s(\mathbb{R}^+))\cap X^{s,b}(\R^+\times (0,T)),
\end{equation*}
with its norm   defined as
\begin{equation*}
    \|v\|_{Y_{s,T}}=\sup_{t<T} \|v(\cdot,t)\|_{H^s(\mathbb{R}^+)}+\|v(x,t)\|_{X^{s,b}(\R^+\times (0,T))}.
\end{equation*}

\begin{proof}
(\textbf{Theorem \ref{theorem}})
For $T>0$, we write
\begin{align*}
u(x,t)=\eta(t)W_\R(\varphi^*,\psi^*)&+\eta(t)W_{bdr}(\vec{h}-\vec{p})\\&+\eta_T(t)
\left(\int^t_0 [W_\R(0,g)](x,t-\tau)d\tau-W_{bdr}(\vec{q})\right),
\end{align*}
where $g=\eta_T(t) u^2$, $\vec{p}$ as defined in Lemma \ref{nonf} and
$\vec{q}=\eta_T(t)  (q_1,q_2,q_3)$ with $q_j$ for $j=1,2,3$ defined in Lemma \ref{force}.
 It is can be verified
that $u= u(x,t)$ solves the IBVP in $[0,T]$ for $T<1$.
For given $(\varphi ,\psi ,\vec{h})\in Q_{s}$, let $r>0$ and $T>0$ be constants to be determined later. Let
\begin{equation*}
    S_{r,T}=\{u\in Y_{s,T}, \|u\|_{Y_{s,T}}\leq r\},
\end{equation*}
then the set $S_{r,T}$ is a convex, closed and bounded subset of $Y_{s,T}$. Define a map $\Gamma$ on $S_{r,T}$ by
\begin{equation*}
    \Gamma(u)=u(x,t),
\end{equation*}
 for $u\in S_{r,T}$.
We will show that $\Gamma$ is a contraction map from $S_{r,T}$ to $S_{r,T}$ for proper $r$ and $T$.  Applying Lemma \ref{bdr}, \ref{nonf} and \ref{force}, Proposition \ref{1} and \ref{hi} yields that,
\begin{align*}
\|\Gamma(u)\|_{H^s(\R^+)}\lesssim &\|\eta(t)W_\R(\varphi^*,\psi^*)\|_{H_x^s(\R^+)}+\|\eta(t)W_{bdr}(\vec{h}-\vec{p})\|_{H_x^s(\R^+)}\\
&+\left\|\eta_T(t)
\left(\int^t_0 [W_\R(0,u^2)](x,t-\tau)d\tau-W_{bdr}(\vec{q})\right)\right\|_{H_x^s(\R^+)}\\
\lesssim & \|(\varphi,\psi,\vec{h})\|_{Q_s}\\
&+\left\|\eta_T(t)
\left(\int^t_0 [W_\R(0,u^2)](x,t-\tau)d\tau-W_{bdr}(\vec{q})\right)\right\|_{X^{s,a_1}},
\end{align*}
with $a_1= \frac{3-2b}{4}$ and $b=\frac12-$. Recall the fact $X^{s,a_1}\subseteq C^0_t H^s_x$ for $a_1>\frac12$ (c.f. \cite{54,53,52,50}), thus, according to Lemma \ref{a}, \ref{f1} and Proposition \ref{t1},
\begin{align*}
\left\|\eta_T(t)\int^t_0 [W_\R(0,g)](x,t-\tau)d\tau\right\|_{X^{s,a_1}}\lesssim & T^{1-(a_1+b)}\left\|\frac{\xi^2\widehat{g}}{\phi(\xi)}\right\|_{X^{s,-b}}\\
\lesssim &T^{\frac14-\frac{b}{2}}\|u\|^2_{X^{s,b}}.
\end{align*}
 Moreover, according to Lemma \ref{a}, Proposition \ref{2}, \ref{k1}   and  set $a_2=\frac{6b-1}{4}<b$,
\begin{align*}
\left\|\eta_T(t)[W_{bdr}(\vec{q})](x,t)\right\|_{X^{s,a_1}} \lesssim  \|\vec{q}\|_{{\mathcal H}^s} &\lesssim \left\|\frac{\xi^2\widehat{g}}{\phi(\xi)}\right\|_{X^{s,-a_2}}\lesssim \|\eta_T u\|_{X^{s,a_2}}\| u\|_{X^{s,a_2}}\\
&\lesssim T^{b-a_2}\|u\|^2_{X^{s,b}}= T^{\frac14-\frac{b}{2}}\|u\|^2_{X^{s,b}}.
\end{align*}
Hence,
\[\|\Gamma(u)\|_{H^s(\R^+)}\leq C\|(\varphi,\psi,\vec{h})\|_{Q_s}+ 2CT^{\frac14-\frac{b}{2}}r^2.\]
Furthermore, we have,
\begin{align*}
\|\Gamma(u)\|_{X^{s,b}}\lesssim &\|\eta(t)W_\R(\varphi^*,\psi^*)\|_{X^{s,b}}+\|\eta(t)W_{bdr}(\vec{h}-\vec{p})\|_{X^{s,b}}\\
&+\left\|\eta_T(t)
\left(\int^t_0 [W_\R(0,g)](x,t-\tau)d\tau-W_{bdr}(\vec{q})\right)\right\|_{X^{s,b}}\\
\lesssim & \|(\varphi,\psi,\vec{h})\|_{Q_s}\\
&+\left\|\eta_T(t)
\left(\int^t_0 [W_\R(0,g)](x,t-\tau)d\tau-W_{bdr}(\vec{q})\right)\right\|_{X^{s,b}}.
\end{align*}
Again, according to Proposition \ref{k1}, \ref{t1} and \ref{t2}, Lemma \ref{f1}, one has,
\begin{align*}
\left\|\eta_T(t)
\int^t_0 [W_\R(0,g)](x,t-\tau)d\tau \right\|_{X^{s,b}}\lesssim & T^{1-b-a_2}\left\|\frac{\xi^2\widehat{g}}{\phi(\xi)}\right\|_{X^{s,-a_2}}\\
\lesssim& T^{\frac54-\frac{5b}{2}}T^{\frac14-\frac{b}{2}}\|u\|^2_{X^{s,b}},
\end{align*}
and
\[
\left\|\eta_T(t)[W_{bdr}(\vec{q})](x,t)\right\|_{X^{s,b}} \lesssim  \|\vec{q}\|_{{\mathcal H}^s} \lesssim  T^{\frac14-\frac{b}{2}}\|u\|^2_{X^{s,b}}.
\]
Thus,
\[\|\Gamma(u)\|_{X^{s,b}}\leq C\|(\varphi,\psi,\vec{h})\|_{Q_s}+ C(T^{\frac32-3b}+T^{\frac14-\frac{b}{2}})r^2,\]
and we can choose $T$ small enough depending on $r$  such that, \[r=2C\|(\varphi,\psi,\vec{h})\|_{Q_s},\quad (T^{\frac32-3b}+T^{\frac14-\frac{b}{2}})r\leq \frac12, \quad 2T^{\frac14-\frac{b}{2}}r\leq \frac12\] then,
\[\|\Gamma(u)\|_{C(0,T;H^s(\R^+))\cap X^{s,b}(\R^+\times(0,T))}\leq r.\]
Similar estimates can be drawn for $\Gamma(u)-\Gamma(v)$, therefore for such $r$ and $T$  the map $u=\Gamma(u)$ is contraction.
\end{proof}

%
%
%
%
%
%
%


\begin{thebibliography}{11}

\bibitem{4}
J.~L. Bona, M.~Chen,
\newblock A Boussinesq system for two-way propagation of nonlinear dispersive
  waves.
\newblock {\em Physica D: Nonlinear Phenomena}. 116 (1998) 191--224.

\bibitem{6}
J.~L. Bona, M.~Chen, J. Saut,
\newblock Boussinesq equations and other systems for small-amplitude long waves
  in nonlinear dispersive media. ii. the nonlinear theory.
\newblock {\em Nonlinearity}. 17 (2004) 925--952.

\bibitem{7}
J.~L. Bona, M.~Chen,   J.~Saut,
\newblock Boussinesq equations and other systems for small-amplitude long waves
  in nonlinear dispersive media. I. derivation and linear theory.
\newblock {\em Journal of Nonlinear Science}. 12 (2002) 283--318.

\bibitem{12}
J.~L. Bona,  R.~L. Sachs,
\newblock Global existence of smooth solutions and stability of solitary waves
  for a generalized Boussinesq equation.
\newblock {\em Communications in Mathematical Physics}. 118 (1988) 15--29.

\bibitem{17}
J.~L. Bona, S.~M. Sun,  B.-Y. Zhang,
\newblock A non-homogeneous boundary-value problem for the Korteweg-de Vries
  equation in a quarter plane.
\newblock {\em Transactions of the American Mathematical Society}.
  354 (2002) 427--490.



\bibitem{bsz-2} J. L. Bona, S.-M. Sun, B.-Y. Zhang,
A nonhomogeneous boundary-value problem for the Korteweg-de Vries
equation in a bounded domain, \emph{Comm. Partial Diff.
Equations.} 28 (2003) 1391--1436.


\bibitem{bsz-5}  J. L. Bona, S.-M. Sun,  B.-Y. Zhang,
 Boundary smoothing properties of the Korteweg-de Vries
equation in a quarter plane and applications,\emph{ Dynamics Partial
Diff. Equations}. 3 (2006) 1--70. 

\bibitem{bsz-6}  J. L. Bona, S.-M. Sun,  B.-Y. Zhang,
 Nonhomogeneous problems for the Korteweg-de Vries and
the Korteweg-de Vries-Burgers equations in a quarter plane,  \emph{Ann.
Inst. H. Poincar\'e, Anal. Non Lin\'eaire}. 25 (2008)
1145--1185.



\bibitem{18}
J.~L. Bona, S.~M. Sun, B.-Y. Zhang,
\newblock A non-homogeneous boundary-value problem for the Korteweg-de Vries
  equation posed on a finite domain. II.
\newblock {\em Journal of Differential Equations}. 247 (2009) 2558--2596.

\bibitem{23}
J.~Bourgain,
\newblock Fourier transform restriction phenomena for certain lattice subsets
  and applications to nonlinear evolution equations. I. Schr\"odinger equations.
\newblock {\em Geometric and Functional Analysis}. 3 (1993) 107--156.

\bibitem{24}
J.~Bourgain,
\newblock Fourier transform restriction phenomena for certain lattice subsets
  and applications to nonlinear evolution equations. II. the KdV-equation.
\newblock {\em Geometric and Functional Analysis}. 3 (1993) 209--262.


\bibitem{28}
C.~Christov, M.~Velarde,
\newblock Well-posed Boussinesq paradigm with purely spatial higher-order
  derivatives.
\newblock {\em Phys.Rev.E}. 54 (1996) 3621--3638.

\bibitem{97}
E.~Compaan, N.~Tzirakis,
\newblock Well-posedness and nonlinear smoothing for the good Boussinesq
  equation on the half-line.
\newblock {\em Journal of Differential Equations}. 262 (2017) 5824--5859.

\bibitem{colliander} J. E. Colliander, C. E. Kenig, The generalized Korteweg-de Vries equation on the
half line, \emph{Comm. Partial Diff. Equations}. 27 (2002) 2187--2266.

\bibitem{fami-2} A. V. Faminskii, A mixed problem in a semistrip for the
Korteweg-de Vries equation and its generalizations, (Russian)
\emph{Dinamika Sploshn. Sredy}. 258 (1988) 54--94; English
transl. in \emph{Trans. Moscow Math. Soc}. 51 (1989) 53--91.

\bibitem{fami-3} A. V. Faminskii, Mixed problems for the Korteweg-de Vries
equation, \emph{Sbornik: Mathematics}.  190 (1999) 903--935.

\bibitem{32}
J.~de~Frutos, T.~Ortega, J.~M. Sanz-Serna,
\newblock Pseudospectral method for the ``good'' Boussinesq equation.
\newblock {\em Mathematics of Computation}. 57 (1991) 109--122.

\bibitem{95}
M.~Erdo{\u{g}}an, N.~Tzirakis,
\newblock Regularity properties of the cubic nonlinear Schr{\"o}dinger equation
  on the half line.
\newblock {\em Journal of Functional Analysis}. 271 (2016) 2539--2568.

\bibitem{34}
A.~Esfahani, L.~G. Farah,
\newblock Local well-posedness for the sixth-order Boussinesq equation.
\newblock {\em Journal of Mathematical Analysis and Applications}.
  385 (2012) 230--242.

\bibitem{35}
A.~Esfahani, L.~G. Farah,  H.~Wang,
\newblock Global existence and blow-up for the generalized sixth-order
  Boussinesq equation.
\newblock {\em Nonlinear Analysis.Theory, Methods and Applications.An
  International Multidisciplinary Journal.Series A: Theory and Methods}.
  75 (2012) 4325--4338.

\bibitem{37}
A.~Esfahani, H.~Wang,
\newblock A bilinear estimate with application to the sixth-order Boussinesq
  equation.
\newblock {\em Differential Integral Equations}. 27 (2014) 401--414.

\bibitem{40}
Y.-F. Fang, M.~G. Grillakis,
\newblock Existence and uniqueness for Boussinesq type equations on a circle.
\newblock {\em Communications in Partial Differential Equations}.
  21 (1996) 1253--1277.

\bibitem{41}
L.~G. Farah,
\newblock Local solutions in Sobolev spaces with negative indices for the
  ``good'' Boussinesq equation.
\newblock {\em Communications in Partial Differential Equations}.
  34 (2009) 52--73.

\bibitem{42}
L.~G. Farah, M.~Scialom,
\newblock On the periodic ``good'' Boussinesq equation.
\newblock {\em Proceedings of the American Mathematical Society}.
  138 (2010) 953--964.

  \bibitem{holmer} J. Holmer, The initial-boundary value problem for the $1$-$d$ nonlinear Schr\"odinger
equation on the half-line. \emph{Diff. Integral Equations}.
18 (2005) 647--668.

\bibitem{holmer-2} J. Holmer, The initial-boundary value problem for the Korteweg-de Vries equation,
\emph{Commun. Partial Diff. Equations}. 31 (2006) 
1151--1190.




\bibitem{54}
C.~E. Kenig, G.~Ponce, L.~Vega,
\newblock Well-posedness of the initial value problem for the Korteweg-de Vries
  equation.
\newblock {\em Journal of the American Mathematical Society}. 4 (1991) 323--347.

\bibitem{53}
C.~E. Kenig, G.~Ponce, L.~Vega,
\newblock The Cauchy problem for the Korteweg-de Vries equation in Sobolev
  spaces of negative indices.
\newblock {\em Duke Mathematical Journal}. 71 (1993) 1--21.

\bibitem{52}
C.~E. Kenig, G.~Ponce, L.~Vega,
\newblock Well-posedness and scattering results for the generalized Korteweg-de
  Vries equation via the contraction principle.
\newblock {\em Communications on Pure and Applied Mathematics}. 46 (1993) 527--620.

\bibitem{50}
C.~E. Kenig, G.~Ponce, L.~Vega,
\newblock A bilinear estimate with applications to the KdV equation.
\newblock {\em Journal of the American Mathematical Society}. 9 (1996) 573--603.

\bibitem{51}
C.~E. Kenig, G.~Ponce, L.~Vega,
\newblock Quadratic forms for the $1$-D semilinear Schr\"odinger equation.
\newblock {\em Transactions of the American Mathematical Society}.
  348 (1996) 3323--3353.

\bibitem{96}
S.~Li, M.~Chen, B.~Zhang,
\newblock A non-homogeneous boundary value problem of the sixth order
  Boussinesq equation in a quarter plane.
\newblock {\em Discrete and Continuous Dynamical Systems - Series A}. 38 (2018) 2505--2525.

\bibitem{98}
S.~Li, M.~Chen, B.~Zhang,
\newblock  Wellposedness of the sixth order Boussinesq equation with non-homogeneous boundary value on a bounded domain.
\newblock {\em Accepted by Physica D: Nonlinear Phenomena}

\bibitem{59}
F.~Linares,
\newblock Global existence of small solutions for a generalized Boussinesq
  equation.
\newblock {\em Journal of Differential Equations}. 106 (1993) 257--293.

\bibitem{60}
F.-L. Liu and D.~L. Russell,
\newblock Solutions of the Boussinesq equation on a periodic domain.
\newblock {\em Journal of Mathematical Analysis and Applications}.
  194 (1995) 78--102.

\bibitem{64}
Y.~Liu,
\newblock Instability of solitary waves for generalized Boussinesq equations.
\newblock {\em Journal of Dynamics and Differential Equations}. 5 (1993) 537--558.

\bibitem{63}
Y.~Liu,
\newblock Instability and blow-up of solutions to a generalized Boussinesq
  equation.
\newblock {\em SIAM Journal on Mathematical Analysis}. 26 (1995) 1527--1546.

\bibitem{62}
Y.~Liu,
\newblock Decay and scattering of small solutions of a generalized Boussinesq
  equation.
\newblock {\em Journal of Functional Analysis}. 147 (1997) 51--68.

\bibitem{61}
Y.~Liu,
\newblock Strong instability of solitary-wave solutions of a generalized
  Boussinesq equation.
\newblock {\em Journal of Differential Equations}. 164 (2000) 223--239.

\bibitem{66}
G.~A. Maugin,
\newblock {\em Nonlinear waves in elastic crystals}.
\newblock Oxford University Press, Oxford, 1999.

\bibitem{67}
S.~Oh, A.~Stefanov,
\newblock Improved local well-posedness for the periodic ``good'' Boussinesq
  equation.
\newblock {\em Journal of Differential Equations}. 254 (2013) 4047--4065.


\bibitem{75}
R.~L. Sachs,
\newblock On the blow-up of certain solutions of the ``good'' Boussinesq
  equation.
\newblock {\em Applicable Analysis.An International Journal}. 36 (1990) 145--152.

\bibitem{76}
M.~Tsutsumi, T.~Matahashi,
\newblock On the Cauchy problem for the Boussinesq type equation.
\newblock {\em Mathematica Japonica}. 36 (1991) 371--379.

\bibitem{79}
H.~Wang, A.~Esfahani,
\newblock Well-posedness for the cauchy problem associated to a periodic
  Boussinesq equation.
\newblock {\em Nonlinear Analysis: Theory, Methods and Applications}.
  89 (2013) 267--275.

\bibitem{84}
R.~Xue,
\newblock Local and global existence of solutions for the Cauchy problem of a
  generalized Boussinesq equation.
\newblock {\em Journal of Mathematical Analysis and Applications}.
  316 (2006) 307--327.

\bibitem{82}
R.~Xue,
\newblock The initial-boundary value problem for the ``good'' Boussinesq
  equation on the bounded domain.
\newblock {\em Journal of Mathematical Analysis and Applications}.
  343 (2008) 975--995.

\bibitem{83}
R.~Xue,
\newblock The initial-boundary-value problem for the ``good'' Boussinesq
  equation on the half line.
\newblock {\em Nonlinear Analysis.Theory, Methods and Applications.An
  International Multidisciplinary Journal.Series A: Theory and Methods}.
  69 (2008) 647--682.

\bibitem{81}
R.~Y. Xue,
\newblock Low regularity solution of the initial-boundary-value problem for the
  ``good'' Boussinesq equation on the half line.
\newblock {\em Acta Mathematica Sinica (English Series)}. 26 (2010) 2421--2442.

\bibitem{85}
Z.~Yang,
\newblock On local existence of solutions of initial boundary value problems
  for the ``bad'' Boussinesq-type equation.
\newblock {\em Nonlinear Analysis.Theory, Methods and Applications.An
  International Multidisciplinary Journal.Series A: Theory and Methods}.
  51 (2002) 1259--1271.
\end{thebibliography}
\end{document}